\documentclass[]{article}

\pdfoutput=1

\usepackage{amsfonts}
\usepackage{amsmath}
\usepackage{amsthm}
\usepackage{amssymb}
\usepackage{graphicx}
\usepackage{epstopdf}
\usepackage{verbatim}
\usepackage{float}
\usepackage{tikz}
\usetikzlibrary{matrix,arrows,decorations.pathmorphing}
\usepackage{tikz-cd}
\usepackage[all,cmtip]{xy}

\theoremstyle{definition}
\newtheorem{defn}{Definition}[section]
\newtheorem{ex}{Example}[section]

\theoremstyle{plain}
\newtheorem{theorem}{Theorem}[section]
\newtheorem{lemma}{Lemma}[section]
\newtheorem{cor}{Corollary}[section]
\newtheorem{prop}{Proposition}[section]

\DeclareMathOperator{\Rm}{Rm}
\DeclareMathOperator{\End}{End}

\DeclareMathOperator{\Span}{Span}

\DeclareMathOperator{\Sym}{Sym}
\DeclareMathOperator{\Id}{Id}

\DeclareMathOperator{\SL}{SL}

\DeclareMathOperator{\Vol}{Vol}

\DeclareMathOperator{\length}{length}

\title{Regular ambitoric $4$-manifolds: from Riemannian Kerr to a complete classification}
\author{Kael Dixon}

\begin{document}

\maketitle

\begin{abstract}
 We show that the conformal structure for the Riemannian analogues of Kerr black-hole metrics can be given an ambitoric structure. We then discuss the properties of the moment maps. In particular, we observe that the moment map image is not locally convex near the singularity corresponding to the ring singularity in the interior of the black hole. We then proceed to classify regular ambitoric $4$-orbifolds with some completeness assumptions. The tools developed also allow us to prove a partial classification of compact Riemannian $4$-manifolds which admit a Killing $2$-form.
\end{abstract}
\tableofcontents
\section{Introduction}

The motivating examples for this paper are the Riemannian analogues of the Kerr family of metrics. The Kerr metrics have Lorentzian metrics and are used to model isolated rotating black holes. One of the interesting properties of the Kerr metrics is that their Weyl curvature is algebraically special of type $D$. The Riemannian analogue of this type $D$ condition is that the self- and anti-self dual Weyl curvature tensors $W_\pm$ both have a unique $1$-dimensional eigenspace. By the Riemannian analogue of the Goldberg-Sachs theorem \cite{apostolov1998generalized}, when the metric is additionally Einstein, this condition is equivalent to the existence of two integrable complex structure $J_\pm$ which induce opposite orientations, are compatible with the metric, and are unique upto sign. We say that $(g,J_+,J_-)$ is an \emph{ambihermitian} structure. If an ambihermitian manifold (or more generally orbifold) $(M,g,J_+,J_-)$ admits the action of a $2$-torus $\mathbb T$ and metrics $g_\pm$ in the conformal class $[g]$ such that $(M,g_\pm,J_\pm,\mathbb T)$ are both K\"ahler toric orbifolds, then $(M,[g],J_+,J_-,\mathbb T)$ is \emph{ambitoric}. We show that the Riemannian Kerr metrics admit a regular ambitoric structure. See definition \ref{defRegAmbi} for the definition of regular.

The main purposes of this paper are to understand the momentum map of the Kerr examples and to classify regular ambitoric $4$-orbifolds under some completeness assumptions.

One of the tools for studying toric manifolds is the moment map, which in the case of compact toric manifolds identifies the space of torus orbits with a convex polytope in the dual Lie algebra of the torus \cite{atiyah1984moment}. However, in the case of the Kerr ambitoric structure, we find that the moment map images are not convex (even locally) in the regions corresponding to the ring singularity in the interior of the black hole.

To understand the moment map image of Kerr, we use the language of \emph{folded symplectic structures} studied in \cite{da2011symplectic}. Folded symplectic structures are $2$-forms which are closed and non-degenerate away from certain hypersurfaces which are called \emph{folding hypersurfaces}. We interpret the K\"ahler forms of the Kerr ambitoric structure as folded symplectic structures on manifolds which have folding hypersurfaces where the local convexity of the moment map fails. This result can be understood in the context of Example 3.11 in \cite{da2011symplectic}, which shows that there is no reason for the moment map to be locally convex near a folding hypersurface.

In the second half of the paper, we classify complete regular ambitoric orbifolds. We find it convenient to work with slight generalizations of regular ambitoric orbifolds which we call \emph{regular ambitoric orbifold completions}, for which the full ambitoric structure may only be defined on a dense open subset. More precisely, we call a Riemannian manifold \emph{completable} if its Cauchy completion is an orbifold. A regular ambitoric orbifold completion is defined to be the Cauchy completion of a completable regular ambitoric manifold equipped with a free torus action. Thus every complete regular ambitoric orbifold is a regular ambitoric orbifold completion, since the set of its free torus orbits is an open dense submanifold.

Regular ambitoric $4$-orbifolds have been locally classified on the set of orbits where the torus action is free \cite{apostolov2013ambitoric}. We use this local classification to generate a set of examples of regular ambitoric $4$-manifolds with free torus action, which we call \emph{ambitoric ansatz spaces}. These examples include the free orbits of the Riemannian Kerr examples. These ambitoric ansatz spaces depend on a symmetric quadratic polynomial $q(x,y)$ and two functions of one variable $A(x)$ and $B(y)$, with the manifold given by
 \begin{align*}
 &\mathcal A(q,A,B,\mathbb T):=\\&
 \bigg\{\big(x,y,\vec t\big)\in\mathbb R^2\times \mathbb T:
 A(x)>0,B(y)>0,q(x,y)(x-y)\neq 0\bigg\},
 \end{align*}
 equipped with the ambitoric structure given in theorem \ref{thmAmbitoric}.

Note that an ambitoric structure gives a conformal class of metrics. We will want to study metric properties, so we need to make a choice of metric in this conformal class. We pick out some distinguished metrics to study, which are unique up to homothety: the K\"ahler metrics $g_\pm$, and the barycentric metric $g_0$, which is an average of $g_+$ and $g_-$. We also study compatible metrics whose Ricci tensor is diagonal, meaning that it is invariant under both $J_+$ and $J_-$. These include the Einstein metrics, and they come in a family $\{g_p\}_{p\perp q}$, indexed by quadratic polynomials $p$ orthogonal to the quadratic polynomial $q$ with respect to the inner product induced by the discriminant.

The first step in our classification is to classify ambitoric ansatz spaces which are completable with respect to our chosen metric $g\in\{g_0,g_\pm\}\cup\{g_p\}_{p\perp q}$. The strategy taken is to first compare the $g$-Cauchy completion to the Cauchy completion with respect to a more convenient metric where the Cauchy completion is easy to compute. This comparison is done using the tools of Busemann completions, the details of which are provided in an appendix. We then decompose the convenient boundary into components which we describe as \emph{folds/corners/edges} according to their behaviour under the moment map. When studying $g_p$, we also have to study the vanishing locus of $p(x,y)$, which we denote by $P$. We further define a component of the convenient boundary is \emph{infinitely distant} if it does not lie in the $g$-Cauchy completion, and a fold is \emph{proper} if it is not an edge or a corner. We distinguish between folds being \emph{positive/negative} in such a way that proper positive/negative folds are folding hypersurfaces for suitable extensions of the positive/negative K\"ahler forms. By studying the $g$-Cauchy completions, we get the following classification result:

\begin{theorem}\label{thmClassCompletable}
	Let $\mathcal A$ be an ambitoric ansatz space and $g\in\{g_\pm,g_0\}\cup\{g_p\}_{p\perp q}$. Then $\mathcal A$ is $g$-completable if and only if the following conditions hold:
	\begin{itemize}
		\item $\mathcal A$ has no proper folds.
		\item Every edge is either infinitely distant or has a compatible normal.
		\item If an edge is a fold but not infinitely distant, then $g\in\{g_-,g_p\}$.
		\item Every corner is infinitely distant unless it is at the intersection of two edges which are not infinitely distant. If such a corner is a positive (respectively negative) fold, then $g\in\{g_+,g_p\}$ (respectively $g\in\{g_-,g_p\}$). If $g=g_p$, then such a corner is not part of $P$.
	\end{itemize}
	Moreover, if $\mathcal A$ is $g$-completable, then its ambitoric structure extends to $\mathcal A_C^{g}$ if and only if every fold is infinitely distant.
\end{theorem}

To classify regular ambitoric orbifold completions, we first extend the local classification on the set of free torus orbits to a local classification on the Cauchy completions. This allows us to describe the boundary of the set of free orbits in terms of folds/edges/corners as was done in the case of ambitoric ansatz spaces. We find that in the case where there are no proper folds, we can apply a slight modification of the Lokal-global-Prinzip for convexity theorems \cite{hilgert1994coadjoint}, which is a tool that can be used to prove that the moment map of a convex toric orbifold is a convex polytope. This gives us the following partial classification of regular ambitoric orbifold completions:

\begin{theorem}\label{thmCompletionClassification}
	 Let $(M,\mathring M,g,J_+,J_-,\mathbb T)$ be a connected regular ambitoric orbifold completion, where $g\in\{g_0,g_\pm\}\cup\{g_p\}_{p\perp q}$. Then there exists an ambitoric embedding of $(\mathring M,[g],J_+,J_-,\mathbb T)$ into some ambitoric ansatz space which is completable with respect to the metric induced by $g$ and whose moment map images are polygons.
\end{theorem}

  We combine this classification with our classification of completable ambitoric ansatz spaces to obtain an explicit classification of $g_0$-complete regular ambitoric $4$-orbifolds:
\begin{cor}\label{corAmbiClass}
	A regular ambitoric $4$-orbifold which is complete with respect to the barycentric metric $g_0$ is given (uniquely upto gauge transformation) by the data of a symmetric quadratic polynomial $q(x,y)$, a pair of intervals \newline $(x_-,x_+),(y_-,y_+) \subset\mathbb R$ such that $(x-y)q(x,y)$ does not vanish on \newline $(x_-,x_+)\times(y_-,y_+)$, a lattice $\Lambda\subset\mathbb R^2$ and two smooth positive functions \newline $A:(x_-,x_+)\to [0,\infty)$ and $B:(y_-,y_+)\to[0,\infty)$ satisfying for all $\epsilon>0$ small enough:
	\begin{itemize}
		\item If $\int_{x_\pm}^{x_\pm\mp\epsilon}\frac{dx}{\sqrt{A(x)}}$ converges, then $-2\frac{p^{(x\pm)}}{A'(x_\pm)}\in\Lambda,$
		\item If $\int_{y_\pm}^{y_\pm\mp\epsilon}\frac{dy}{\sqrt{B(y)}}$ converges, then $-2\frac{p^{(y\pm)}}{B'(y_\pm)}\in\Lambda,$
		\item For each $\alpha,\beta\in\{\pm\}$, if the integrals from the previous conditions corresponding to $x_\alpha$ and $y_\beta$ are both convergent, then $(x_\alpha-y_\beta)q(x_\alpha,y_\beta)\neq 0$,
	\end{itemize}
\end{cor}
where $p^{(\gamma)}\in\mathfrak t$ is a convenient choice of normal to the momentum map image of the level sets $\{x=\gamma\}$ and $\{y=\gamma\}$. See definition \ref{defCompNormal}. Note that similar results can be obtained for $g_\pm$, but the third condition which tests the corners would be more complicated.

As an application for our classification results, we consider compact Riemannian $4$-manifolds which admit $*$-Killing $2$-forms. These are Riemannian signature analogues of the Killing-Yano tensor in Lorentzian signature, which describes the so-called hidden symmetries of the Kerr metric. We build on the work of \cite{gauduchon2015killing}, who divide Riemannian $4$-manifolds admitting non-parallel $*$-Killing $2$-forms into $3$ types. They show that one of these types is consists of regular ambitoric orbifold completions, identifying the metric with what we've been calling $g_p$ with $p=1$. We classify compact $4$-manifolds of this type:

\begin{theorem}\label{thmIntroKilling}
	Let $(M,g)$ be a compact connected oriented regular ambitoric orbifold completion which is a manifold admitting a non-parallel $*$-Killing $2$-form. Then $M$ is diffeomorphic to either $\mathbb S^4, \mathbb{CP}^2$, or a Hirzebruch surface. Conversely, each of these manifolds admit a metric with a non-parallel $*$-Killing $2$-form.
\end{theorem}

\subsection{Ambitoric structures and ambitoric ansatz spaces}

In this section, we introduce the basic definitions of ambitoric manifolds, as well as the local classification in four dimensions from \cite{apostolov2013ambitoric}.

\begin{defn}
 A \emph{Hermitian} manifold $(M,g,J)$ is a manifold $M$ equipped with a metric $g$ and a compatible complex structure $J$. ie. $g(J\cdot,J\cdot)=g(\cdot,\cdot)$.
 An \emph{ambihermitian} manifold $(M,g,J_+,J_-)$ is a manifold $M$ equipped with a metric $g$ and complex structures $J_\pm$ such that $(M,g,J_\pm)$ are both Hermitian, and the orientations induced by $J_\pm$ are opposite. If the conformal class $[g]$ of an ambihermitian manifold $(M,g,J_+,J-)$ admits metrics $g_\pm\in[g]$ such that $(M,g_\pm,J_\pm)$ are both K\"ahler manifolds, then $(M,[g],J_+,J_-)$ is \emph{ambik\"ahler}. If there is a $2$-torus $\mathbb T$ which acts on an ambik\"ahler manifold $(M,[g],J_+,J_-)$ by hamiltonial isometries, then $(M,[g],J_+,J_-,\mathbb T)$ is an \emph{ambitoric} manifold. Note that all of these definitions extend naturally to the case of orbifolds.
\end{defn}

We denote the Lie algebra of $\mathbb T$ by $\mathfrak t$. The kernel of the exponential map on $\mathfrak t$ is the lattice of circle subgroup $\Lambda\subset\mathfrak t$, so that $\mathbb T \cong \mathfrak t/\Lambda$.

For an ambitoric manifold, $(M,[g],J_+,J_-,\mathbb T)$, we denote by $\omega_\pm:=g_\pm(J_\pm\cdot,\cdot)$ the K\"ahler forms of $(M,g_\pm,J_\pm)$.
Since $g_\pm$ are in the same conformal class, there exists a positive function $f$ such that $g_-=f^2 g_+$.

\begin{defn}
 The \emph{barycentric metric} on an ambiK\"ahler manifold is given by $g_0:=f g_+=f^{-1}g_-$.
\end{defn}

Note that the metrics $g_\pm$ are uniquely chosen within their conformal class up to homothety, so that $f$ is well-defined up to a multiplicative constant, and $g_0$ is well-defined up to homothety.

Since each $K\in\mathfrak t$ is Hamiltonian with respect to $\omega_\pm$, there exist functions $f_K^\pm\in\mathcal C^\infty(M)$ such that $K\lrcorner\omega_\pm=-df_K^\pm$. The map $K\mapsto f_K^\pm$ is linear, so it gives an element of $\mathfrak t^*\otimes \mathcal C^\infty(M)$. This gives a smooth map $\mu^\pm:M\to\mathfrak t^*$, which is the \emph{moment map} for the toric manifold $(M,\omega_\pm,\mathbb T)$.

Ambitoric $4$-manifolds come in three families \cite{apostolov2013ambitoric}, which we will describe in the rest of the section.

\begin{ex}
 Let $(\Sigma_1,g_1,J_1)$ and $(\Sigma_2,g_2,J_2)$ be two (K\"ahler) Riemann surfaces with non-vanishing Hamiltonian Killing vector fields $K_1$ and $K_2$ respectively. Then $\Sigma_1\times\Sigma_2$ can be given ambitoric structure
 \begin{align*}
  g &:= g_1 \oplus g_2, \\
  J_\pm &:= J_1 \oplus\pm J_2,\\
  \mathfrak t &:=\Span\{K_1,K_2\}.
 \end{align*}
\end{ex}

\begin{ex}
 Let $(M,g,J)$ be a K\"ahler surface. For any non-vanishing hamiltonian Killing vector field $K$, we can define an almost complex structure $J_-$ by
 \begin{equation*}
  J_-:=\left\{\begin{array}{ll}
  J &  \text{on }\Span\{K,JK\}\\
  -J & \text{on }\Span\{K,JK\}^\perp
  \end{array}\right.
 \end{equation*}
 $J_-$ has opposite orientation to $J$. If $(M,g,J)$ is conformally K\"ahler, then $(M,[g],J_+:=J,J_-)$ is ambik\"ahler. Such an ambik\"ahler manifold is said to be of \emph{Calabi type}. In \cite{apostolov2003geometry}, it is shown that the K\"ahler quotient of $(M,g,J)$ using the momentum map $z$ of $K$ is a Riemann surface $(\Sigma,(az-b)g_\Sigma,J_\Sigma)$, where $a,b$ are constants, $g_\Sigma$ and $J_\Sigma$ are a metric and complex structure respectively on $\Sigma$. An ambik\"ahler surface of Calabi type $(M,[g],J,J_-)$ is ambitoric if and only if $(\Sigma,(az-b)g_\Sigma,J_\Sigma)$ admits a Hamiltonian Killing vector field.
\end{ex}

For an ambitoric $4$-manifold $(M,[g],J_+,J_-,\mathbb T)$, we define $\mathfrak t_M$ to be the subset $TM$ spanned by the Killing vector fields $\mathfrak t$. On an open dense subset $\mathring M\subseteq M$, $\mathfrak t_{\mathring M}:=\mathfrak t_M|_{\mathring M}$ is a two dimensional distribution in $T\mathring M$. Since the Killing vector fields $\mathfrak t$ are $\omega_\pm$-Hamiltonian, $\mathfrak t_{\mathring M}$ is $\omega_\pm$-Lagrangian, so that $J_+\mathfrak t_{\mathring M}=J_-\mathfrak t_{\mathring M}=\mathfrak t_{\mathring M}^\perp$. 

Since $J_+$ and $J_-$ have opposite orientations, they commute. This implies that the endomorphism $-J_+J_-$ of $TM$ is an involution. Thus $TM$ decomposes into $\pm 1$ eigenbundles of $-J_+J_-$. Let $\xi_{\mathring M}$ and $\eta_{\mathring M}$ be the intersections of $\mathfrak t_{\mathring M}$ with the $+1$ and $-1$ (respectively) eigenbundles of $-J_+J_-$. Since the eigenbundles are $J_\pm$ invariant, while $\mathfrak t_{\mathring M}$ is not, $\xi_{\mathring M}$ and $\eta_{\mathring M}$ must be line bundles.

Let $\boldsymbol K: \mathfrak t\to\Gamma(\mathfrak t_M)$ be the function which maps a vector $X\in\mathfrak t$ to the associated Killing vector field on $M$. Let $\boldsymbol\theta\in\Omega^1(M,\mathfrak t)$ be the $\mathfrak t$-valued one-form defined by the relations $\boldsymbol{\theta}\circ \boldsymbol K = \Id_{\mathfrak t}$ and $\boldsymbol{\theta}|_{\mathfrak t_{\mathring M}^\perp}=0$. Let
\begin{align*}
 \xi:\mathring M\to\mathbb P\mathfrak t:
 p\mapsto\boldsymbol{\theta}((\xi_{\mathring M})_p),\\
 \eta:\mathring M\to\mathbb P\mathfrak t:
  p\mapsto\boldsymbol{\theta}((\eta_{\mathring M})_p).
\end{align*}

If $\xi$ or $\eta$ is constant, then its image is the span of some Killing vector field $K$ which realizes the ambitoric structure as being Calabi type. If neither is constant, then $d\xi\wedge d\eta$ is non-vanishing \cite{apostolov2013ambitoric}.

\begin{defn}\label{defRegAmbi}
 An ambitoric $4$-manifold $(M,[g],J_+,J_-,\mathbb T)$ is \emph{regular} if $d\xi\wedge d\eta$ is non-vanishing on an open dense set.
\end{defn}
It is clear that every ambitoric structure is either regular or of Calabi type. We will find it convenient to fix some notation for quadratic polynomials:

\begin{defn}
	Let $p(z)=p_0\, z^2+2p_1\,z+p_2$ be a quadratic polynomial. We call $p(x,y)=p_0\,xy+p_1\,(x+y)+p_2$ the \emph{polarization} of $p(z)$. We define a $(2,1)$-signature inner product on the space of quadratic polynomials by $$<q,p>:=2q_1p_1-q_2p_0-q_0p_2.$$
\end{defn}

We will use extensively this local classification of regular ambitoric surfaces:

\begin{theorem}[Theorem $3$ from \cite{apostolov2013ambitoric}]\label{thmAmbitoric}
 Let $(M,[g_0],J_+,J_-,\mathbb T)$ be a regular ambitoric $4$-manifold with barycentric metric $g_0$ and K\"ahler metrics $(g_+,\omega_+)$ and $(g_-,\omega_-)$. Then, about any point in an open dense subset of $M$, there are $\mathfrak t$-invariant functions $x,y$, a quadratic polynomial $q(z)$, a $\Lambda^1(M)$-valued quadratic polynomial $d\tau(z)$ orthogonal to $q(z)$, and functions $A(z)$ and $B(z)$ of one variable with respect to which:
 \begin{align}\label{eqnAmbiStructure}
  g_0 =& \frac{dx^2}{A(x)}+\frac{dy^2}{B(y)}
   + A(x)\left(\frac{d\tau(y)}{(x-y)q(x,y)}\right)^2
    + B(y)\left(\frac{d\tau(x)}{(x-y)q(x,y)}\right)^2
    ,\nonumber\\
  \omega_+ =& \frac{dx\wedge d\tau(y)+dy\wedge d\tau(x)}{q(x,y)^2},\qquad
  \omega_= \frac{dx\wedge d\tau(y)-dy\wedge d\tau(x)}{(x-y)^2}.
 \end{align}

 Conversely, for any data as above, the above metric and K\"ahler forms do define an ambitoric K\"ahler structure on any simply connected open set where $\omega_\pm$ are nondegenerate and $g_0$ is positive definite.
\end{theorem}

 The open dense subset of $M$ where this theorem applies is the maximal open set $\mathring M$ where the Killing vector fields $\mathfrak t$ have maximal rank. 
 
 The above theorem motivates us to define a family of examples which we will call \emph{ambitoric ansatz spaces}:
 \begin{align*}
 &\mathcal A(q,A,B,\mathbb T):=\\&
 \bigg\{\big(x,y,\vec t\big)\in\mathbb R^2\times \mathbb T:
  A(x)>0,B(y)>0,q(x,y)(x-y)\neq 0\bigg\},
 \end{align*}
which we equip with the ambitoric structure given by (\ref{eqnAmbiStructure}) while identifying $\mathfrak t$ with the infinitesimal vector fields of the action of $\mathbb T$.

\begin{defn}
	A map between ambitoric oribifolds is \emph{ambitoric} if it preserves all of the ambitoric structure. In particular, it is equivariant under the torus action, holomorphic with respect to both complex structures, and preserves the conformal structure.
\end{defn}

We will use these to rephrase the above theorem in the case that the toric structure comes from a Lie group.

\begin{cor}\label{corLocClass}
 Let $(M,[g_0],J_+,J_-,\mathbb T)$ be a regular ambitoric $4$-manifold freely acted on by a $2$-torus $\mathbb T$. Then for any point $p\in M$, there exists a quadratic polynomial $q(z)$, functions $A(z)$ and $B(z)$, a $\mathbb T$-invariant neighbourhood $U$ of $p$, and an ambitoric embedding $\phi:U\hookrightarrow\mathcal A(q,A,B,\mathbb T)$.
\end{cor}
\begin{proof}
 Since the action of $\mathbb T$ is free, $\mathfrak t$ has maximal rank at each point in $M$. Thus we can apply the previous theorem to find local coordinates on a neighbourhood $U'$ of $p$ which is naturally identified as a subset of some $\mathcal A(q,A,B,\mathbb T)$. Since the ambik\"ahler structure of $M$ is $\mathbb T$-equivariant, these coordinates can naturally be extended to the orbit $U:=\mathbb T\cdot U'$.
\end{proof}


Next finish this section by reviewing some results about special metrics on regular ambitoric orbifolds:

\begin{theorem}[\cite{apostolov2013ambitoric}]\label{thmDiagRicci}
	Let $([g],J_\pm,\mathfrak t)$ be a regular ambitoric structure. Then for any quadratic polynomial $p(z)$ orthogonal to $q(z)$, the metric $\frac{(x-y)q(x,y)}{p(x,y)^2}g_0$ has diagonal Ricci tensor. Any $\mathfrak t$-invariant metric in $[g]$ with diagonal Ricci tensor arises in this way. Such a metric has constant scalar curvature if and only if
	\begin{align*}
	 A(z)=&p(z)\rho(z)+R(z),\\
	 B(z)=&p(z)\rho(z)-R(z),
	\end{align*}
	where $\rho(z)$ is a quadratic polynomial orthogonal to $p(z)$ and $R(z)$ is a quartic polynomial orthogonal to $q(z)p(z)$ (equivalently $(q,R)^{(2)}\perp p$ or, equally, 
	$(p,R)^{(2)}\perp q$). The metric is Einstein when $\rho(z)$ is a multiple of $q(z)$. Here the transvectant $(p,R)^{(2)}$ is the quadratic polynomial defined by $$(p,R)^{(2)}:=p(z)R''(z)-3p'(z)R'(z)+6p''(z)R(z).$$
\end{theorem}

\section{Moment maps and folded symplectic structures}
\subsection{Folded symplectic structures}\label{secMomFold}
In this section, we will extend the symplectic structures on ambtoric ansatz spaces to folded symplectic manifolds, and describe their moment map images.

\begin{defn}\cite{da2011symplectic}
	A closed $2$-form $\omega$ on a $2n$-dimensional manifold $M$ is a \emph{folded symplectic structure} if there exists an embedded hypersurface $Z$ of $M$ such that $\omega$ is non-degenerate on $M\backslash Z$ and $\omega^n=0\neq \omega^{n-1}$ on $Z$. $Z$ is called the \emph{folding hypersurface} or \emph{fold}.
\end{defn}

Let
$$\mathcal F(q,A,B,\mathbb T):=\{(x,y,t_1,t_2)\in\mathbb R^2\times \mathbb T^2:A(x)> 0,B(y)> 0\}\supseteq\mathcal A(q,A,B,\mathbb T).$$
To simplify notation the arguments for $\mathcal F$ will be tacit. Let $Z_\pm:=\{f^{\mp 1}=0\}\subseteq\mathcal F$ and $\mathcal F_\pm=\mathcal F\backslash Z_\mp$.
\begin{prop}
	$(\mathcal F_\pm,\omega_\pm)$ is a folded symplectic manifold with fold $Z_\pm|_{\mathcal F_\pm}$, where $\omega_\pm$ is the $2$-form defined in (\ref{eqnAmbiStructure}).
\end{prop}
\begin{proof}
	Recall that $f=\frac{q(x,y)}{x-y}$.  It is a routine computation to show that $\omega_\pm$ is a closed non-vanishing $2$-form on $\mathcal F_\pm$. Also one can compute
	$$\omega_\pm^2 = \frac{f^{\mp 2}}{A(x)B(y)}dx\wedge d_\pm^cx\wedge dy\wedge d_\pm^c y,$$ 
	so that the vanishing locus of $\omega_\pm^2$ is $Z_\pm$. Noting that $\omega_\pm$ does not vanish along $Z_\pm$ then tells us that $Z_\pm$ is a folding hypersurface for $\omega_\pm$.
\end{proof}

\begin{prop}
	$\mu^\pm(Z_\pm)$ is a (possibly degenerate) conic in $\mathfrak t^*$, which we will denote by $\mathcal C_\pm$.
\end{prop}
\begin{proof}
	This can be done by an easy direct computation using the normal forms given in section \ref{secNorm}. In the hyperbolic case, one finds that
	$$\mu^\pm (Z_\pm)=\{\mu_1^\pm\mu_2^\pm=-\tfrac14\}.$$
	In the elliptic case, one finds that
	$$\mu^\pm (Z_\pm)=\{(\mu_1^\pm)^2+(\mu_2^\pm)^2=1\}.$$
	In the parabolic case, one finds that
	\begin{align*}
	\mu^+(Z_+) =& \{(\mu_1^+)^2=4\mu_2^+\},\\
	\mu^-(Z_-) =& \{(0,\tfrac 12),(0,-\tfrac 12)\}.
	\end{align*}
\end{proof}

Note that the type of the conic matches the type of the ambitoric structure (aside from the degenerate conic $\mu^-(Z_-)$ in the parabolic case), so that for example a hyperbolic ambitoric structure gives a hyperbola as the image of the fold.

The conic (or more precisely, its dual) is used in \cite{apostolov2013ambitoric2} to study ambitoric compactifications. The key use made in \cite{apostolov2013ambitoric2} of the conic is the fact that the moment map sends level sets of $x$ or $y$ to (subsets of) lines tangent to the conic. If the conic is non-degenerate, this leads to moment map images near the fold that look like the example shown in figure \ref{figMomFold}. The moment map is a $2-1$ cover near the conic, with the image folding along the conic and remaining in the exterior of the conic. This behaviour was noted at the end of appendix A in \cite{apostolov2013ambitoric2}, so the previous proposition is not essentially new. 

\begin{figure}\protect\label{figMomFold}
	\centering
	\includegraphics[scale =0.9]{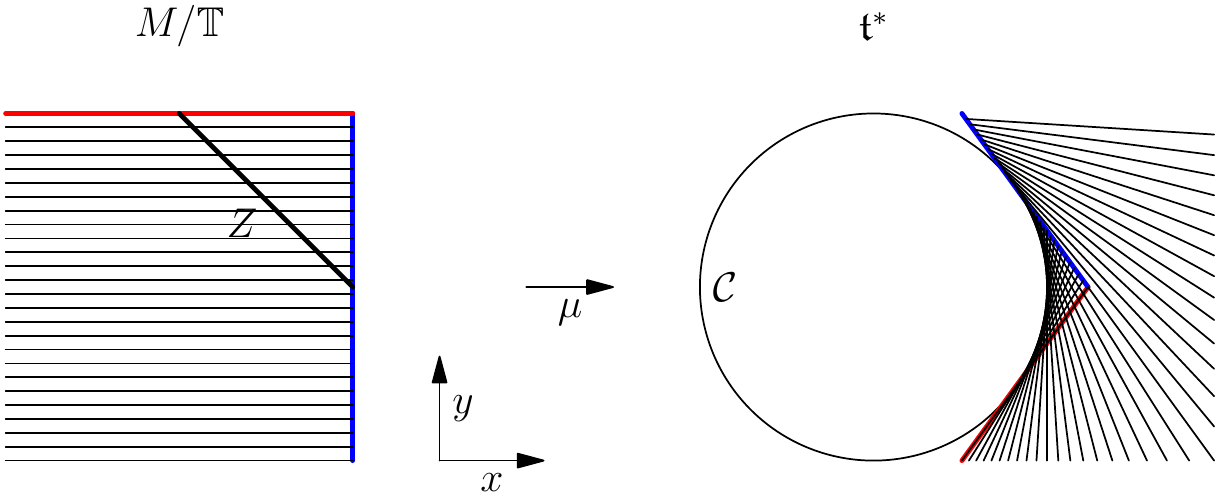}
	\caption{The moment map near a fold $Z$ which intersects level sets of $x$ and $y$ where $A(x)$ and $B(y)$ (shown in blue and red respectively).}
\end{figure}

We will need the following lemma later:

\begin{lemma}\label{lemPMomLine}
	If $p(z)$ is a non-constant quadratic polynomial orthogonal to $q(z)$, then the moment maps $\mu^\pm$ map the vanishing locus of $p(x,y)$ to lines in $\mathfrak t^*$.
\end{lemma}
\begin{proof}
	Consider $p$ as an element of $\mathfrak t$, which is identified with the space of quadratic polynomials orthogonal to $q$. The natural pairings of $p$ with $\mu^\pm(x,y)\in\mathfrak t^*$ are computed in \cite{apostolov2013ambitoric} as
	\begin{align*}
		\mu^+_p(x,y)&=-\frac{p(x,y)}{q(x,y)},\\
		\mu^-_p(x,y)&=-\frac{p(x,y)}{x-y}.
	\end{align*}
	Since $\mu^{\pm}_p(x,y)|_{p(x,y)=0}=0$, we find that the vector $p\in\mathfrak t$ is normal to the curves $\mu^\pm\big(\{p(x,y)=0\}\big)\subset\mathfrak t^*$. This implies that these curves are lines as claimed.
\end{proof}

\subsection{Normal forms}\label{secNorm}
The results in this section are taken directly from \cite{apostolov2013ambitoric}. The classification from theorem \ref{thmAmbitoric} can be further refined into 3 cases, called \emph{parabolic}, \emph{hyperbolic} and \emph{elliptic} respectively if the discriminant of $q(z)$ is zero, negative, or positive respectively. Up to homothety, we can assume that the discriminant is $0$ or $\pm1$, which is done for the normal forms described below. We also provide diagrams indicating the behaviour of the moment maps in the case where $A(x)$ and $B(y)$ are strictly positive functions.

\subsubsection{Parabolic type}
The parabolic type is characterized by $q(z)=1$ and $d\tau_0=0$. This allows us to write the ambitoric structure as
\begin{align*}
g_0 =& \frac{dx^2}{A(x)}+\frac{dy^2}{B(y)}
+\frac{A(x)(dt_1+y\, dt_2)^2}{(x-y)^2}
+\frac{B(y)(dt_1+x\, dt_2)^2}{(x-y)^2},\\
\omega_+ =& {dx\wedge(dt_1+y\, dt_2)}
+ {dy\wedge(dt_1+x\, dt_2)},\\
\omega_- =& \frac{dx\wedge(dt_1+y\, dt_2)}{(x-y)^2}
- \frac{dy\wedge(dt_1+x\, dt_2)}{(x-y)^2}.\\
\end{align*}

the momentum maps $\mu^\pm=(\mu_1^\pm,\mu_2^\pm):\mathcal A\to\mathfrak t^*$ are given by
\begin{align*}
\mu_1^+ =& x+y,\\
\mu_2^+ =& xy,\\
\mu_1^- =& -\frac1{x-y},\\
\mu_2^- =& -\frac{x+y}{2(x-y)}.
\end{align*}

\begin{center}
	\includegraphics[width=0.6\textwidth]{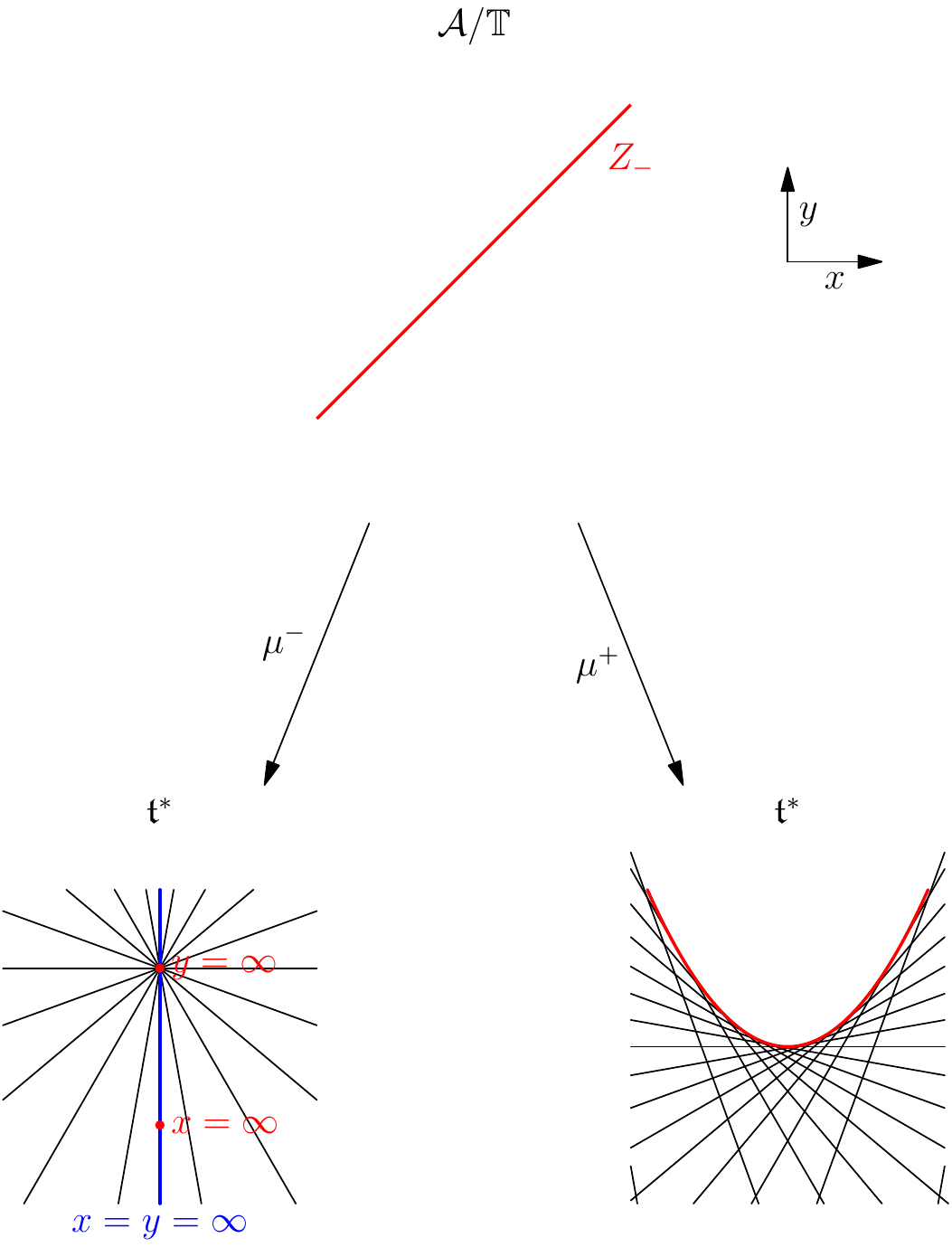} 
\end{center}

\subsubsection{Hyperbolic type}\label{SecHypNorm}
The hyperbolic type is characterized by $q(z)=2z$ and $d\tau_1=0$. This allows us to write the ambitoric structure as
\begin{align*}
g_0 =& \frac{dx^2}{A(x)} + \frac{dy^2}{B(y)}
+ \frac{A(x)(dt_1+y^2\, dt_2)^2}{(x^2-y^2)^2}
+ \frac{B(y)(dt_1+x^2\, dt_2)^2}{(x^2-y^2)^2},\\
\omega_\pm =& \frac{dx\wedge(dt_1+y^2\, dt_2)}{(x\pm y)^2}
\pm\frac{dy\wedge(dt_1+x^2\, dt_2)}{(x\pm y)^2}.
\end{align*}
The momentum map $\mu^\pm=(\mu_1^\pm,\mu_2^\pm):\mathcal A\to\mathfrak t^*$ is given by
\begin{align*}
\mu_1^\pm =& -\frac1{x\pm y}, \\
\mu_2^\pm =& \pm\frac{xy}{x\pm y}.
\end{align*}

\begin{center}
	\includegraphics[width=0.7\textwidth]{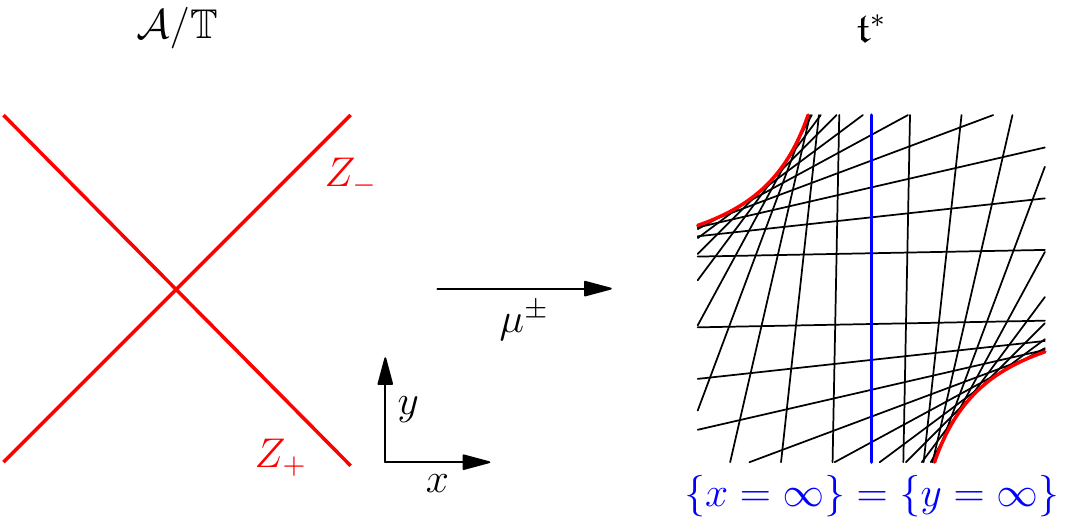} 
\end{center}

\subsubsection{Elliptic type}
The elliptic type is characterized by $q(z)=1+z^2$ and $d\tau_0+d\tau_2=0$. This allows us to write the ambitoric structure as
\begin{align*}
g_0 =& \frac{dx^2}{A(x)}+\frac{dy^2}{B(y)}
+\frac{A(x)(2y\, dt_1+(y^2-1)\, dt_2)^2}{(x-y)^2(1+xy)^2}\\
&+\frac{B(y)(2x\, dt_1+(x^2-1)\, dt_2)^2}{(x-y)^2(1+xy)^2},\\
\omega_+ =& \frac{dx\wedge(2y\, dt_1+(y^2-1)\, dt_2)}{(1+xy)^2}
+ \frac{dy\wedge(2x\, dt_1+(x^2-1)\, dt_2)}{(1+xy)^2},\\
\omega_- =& \frac{dx\wedge(2y\, dt_1+(y^2-1)\, dt_2)}{(x-y)^2}
- \frac{dy\wedge(2x\, dt_1+(x^2-1)\, dt_2)}{(x-y)^2}.\\
\end{align*}

The momentum maps $\mu^\pm=(\mu_1^\pm,\mu_2^\pm):\mathcal A\to\mathfrak t^*$ are given by
\begin{align*}
\mu_1^+ =& -\frac{1-xy}{1+xy},\\
\mu_2^+ =& -\frac{x+y}{1+xy},\\
\mu_1^- =& -\frac{x+y}{x-y},\\
\mu_2^- =& \frac{1-xy}{x-y}.
\end{align*}

\begin{center}
	\includegraphics[width=0.7\textwidth]{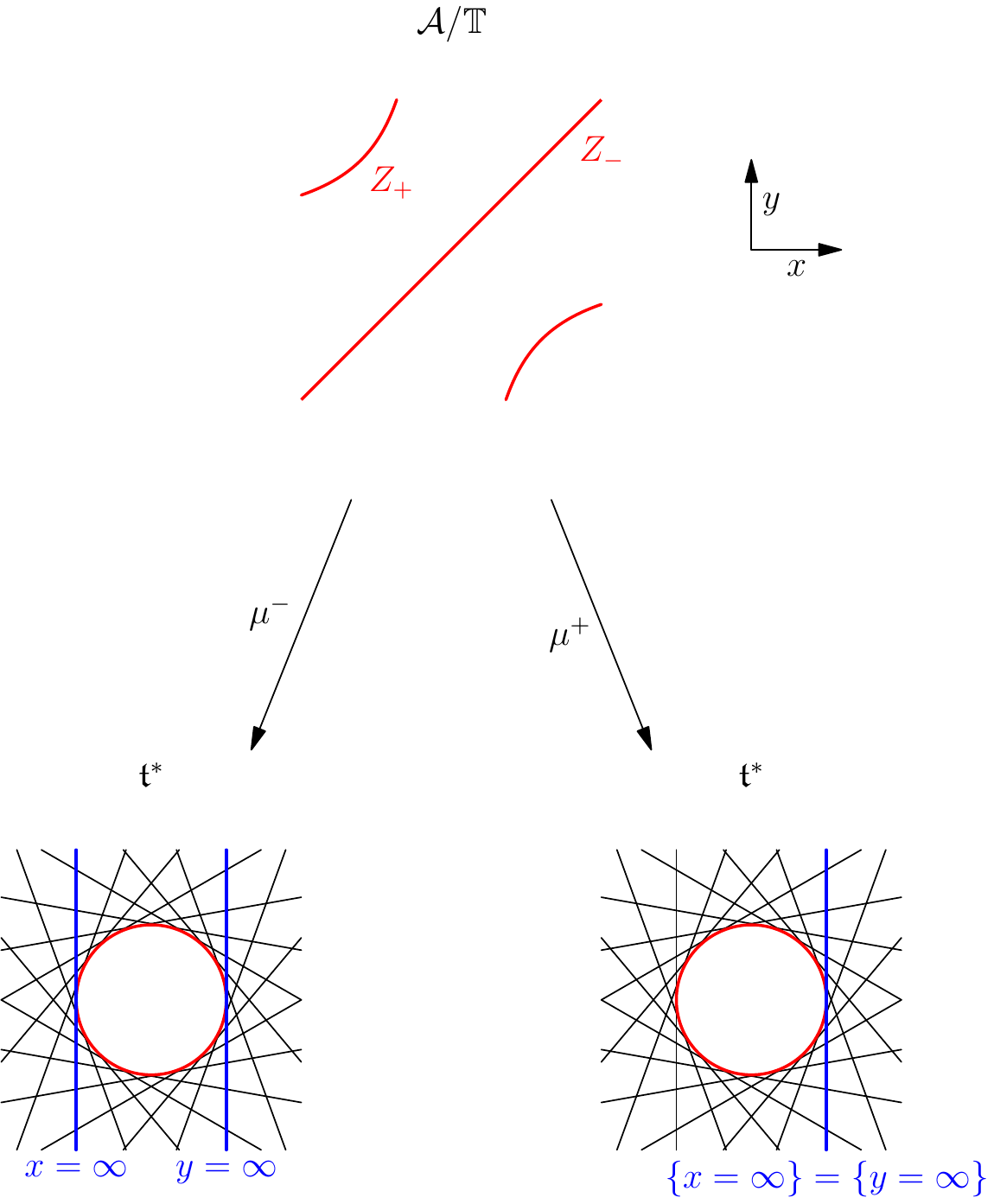} 
\end{center}

\section{Riemannian Kerr metrics are ambitoric}

The Kerr family of metrics are 4-dimensional Lorentzian space-times which are stationary, axis-symmetric, and type D.

The stationary and axis-symmetric condition means that there are commuting Killing vector fields corresponding to the time direction and a rotation around a fixed axis. The span of these vector fields will be our Lie algebra $\mathfrak t$.

The type D condition tells us that the Weyl curvature tensor has two repeated principal null directions. In his work on Hermitian geometry in the Lorentzian setting, Flaherty \cite{flaherty1976hermitian} shows that a metric admits compatible integrable almost complex structures $J_\pm$ with both orientations if and only if the metric is of type D. It is worth noting that in order for an almost complex structure to be compatible with a Lorentzian metric, it must be complex valued (ie, a section of $\End(TM)\otimes\mathbb C$).

Another feature of Kerr that will be useful for us is that it admits a Wick rotation. It is well known in the physics community (see for example \cite{gibbons1979classification}) that if one allows the time and angular momentum variables to take imaginary values instead of real ones, then the metric will still take real values. Moreover, the signature of this Wick-rotated Kerr metric is no longer Lorentzian, but either Riemannian or signature $(2,2)$ depending on the region. We will only consider the regions which have a Riemannian signature, and call the metric in these regions \emph{Riemannian Kerr}. The complex structures $J_\pm$ become real valued after the Wick rotation, so that the Riemannian Kerr metrics are ambihermitian.

Since Einstein ambihermitian spaces are ambik\"ahler \cite{apostolov2013ambitoric}, and $\mathfrak t$ are the infinitesimal generators of a torus action, we see that the Riemannian Kerr metrics are ambitoric. The K\"ahler forms are given in \cite{aliev2006self}. They present the more general Kerr-Taub-bolt instanton, where the construction above will also hold, but for simplicity we only present the Kerr case. The more general case will correspond to adding a parameter that translates the $y$ variable.

After a change of variables from this form, one finds that the Riemannian Kerr fits into the normal form for a hyperbolic type ambitoric manifold, with
\begin{align}
A(x) &= x^2-2Mx-\alpha^2,\\
B(y) &=\alpha^2-y^2,
\end{align}
where $M$ and $\alpha$ are real parameters corresponding to the mass and imaginary angular momentum of the Wick rotated black hole. The Ricci-flat metric is given by $(x^2-y^2)g_0$.

The roots of $A(x)$ are given by
$$x_\pm=M\pm\sqrt{M^2+\alpha^2}.$$
Since $\alpha<M$, the domain in $x-y$ coordinates looks like the following:

\begin{figure}[H]
	\centering
	\includegraphics[scale =0.9]{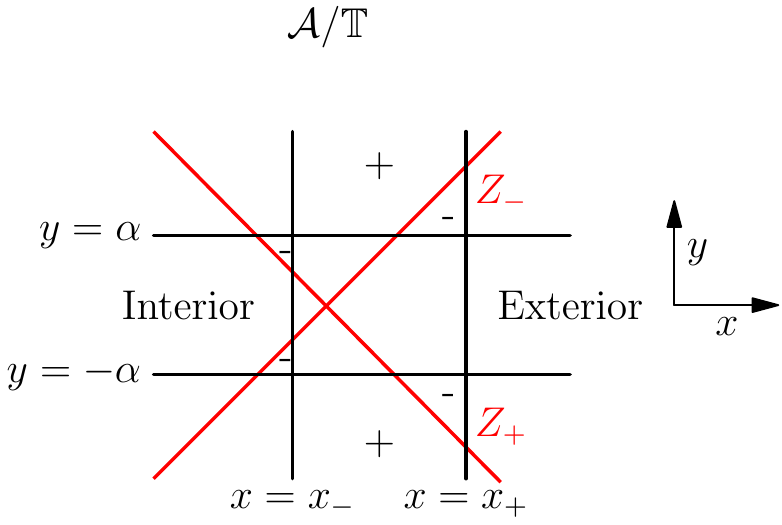}
	\caption{The domain of $\mathcal A/\mathbb T$ in $(x,y)$ coordinates. The regions which we call interior and exterior are labelled, and have positive Riemannian signature. The other regions with $\pm$ Riemannian signature are labelled $\pm$. The remaining regions have signature $(2,2)$.}
\end{figure}

The region corresponding to $x>x_+$ corresponds to the exterior of the black hole. The boundary components of this region correspond to roots of $A(x)$ or $B(y)$ or $x\to\infty$. The region at $x=\infty$ is infinitely distant, so we don't need to extend the structure there. Extending the ambitoric structure to the other boundary regions is well understood, as discussed in lemma \ref{lemEdgeComp}.

The region corresponding to $x<x_-$ corresponds to the interior of the black hole. In Lorentzian signature, there is a ring singularity in the interior region where the curvature diverges to infinity. This singularity transforms under Wick rotation to the folding hypersurfaces $Z_
\pm:=\{x\mp y=0\}$ of Riemannian Kerr. This motivates the work later in this paper where we study the geometry near these hypersurfaces.

The moment map image is roughly the same for both $\omega_\pm$. The main component is drawn in figure \ref{figMomKerr}. There are two triangular domains with edges given by $\{x=x_-\},\{y=\pm\alpha\}$ and $Z_\pm$. One of these regions is near the fold, so has the behaviour depicted in figure \ref{figMomFold}, although it is hard to see in figure \ref{figMomKerr}. The other region (not drawn) gets mapped to a cone with edges given by $\{x=x_-\}$ and $\{y=\mp\alpha\}$ oriented away from the component which was drawn.

\begin{figure}[H]
	\centering
	\includegraphics[scale=1]{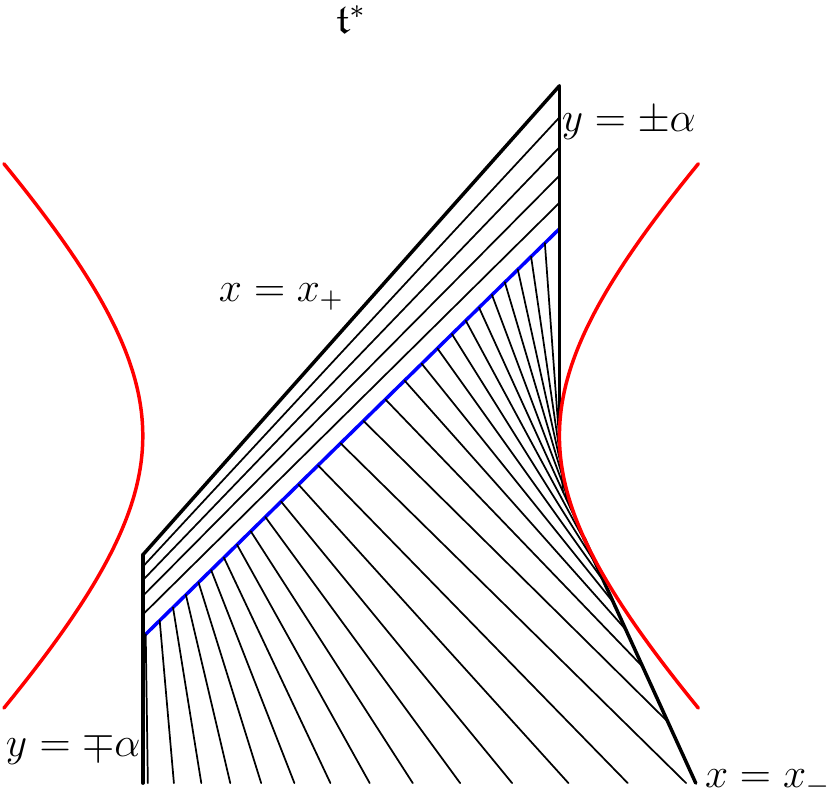}
	\caption{The momentum image of Riemannian Kerr. The conic $\mathcal C$ is drawn in red, and $\{x=\infty\}$ is drawn in blue. The exterior region is shaded with level sets of $x$, while the interior region is shaded with level sets of $y$.}
	\label{figMomKerr}
\end{figure}

\section{Classification of complete regular ambitoric $4$-orbifolds}

In this section, we will classify regular ambitoric $4$ orbifolds with some completeness assumptions. The idea is to focus on the open dense set where the local classification, theorem \ref{thmAmbitoric}, holds, with the intention of recovering the global structure by studying the asymptotic behaviour. In order to get global results from the local classification, we will first study how to glue together the local pieces. This is done by studying the gauge group of regular ambitoric structures.

\subsection{Gauge group for ambitoric ansatz spaces}\label{secGauge}
	
	The local form of a regular ambitoric $4$-manifold given in Theorem \ref{thmAmbitoric} is not unique. To study the flexibility in the local form, let $U$ be a neighbourhood of a point in a regular ambitoric $4$-manifold $(M,[g_0],J_+,J_-,\mathbb T)$ which is small enough so that the theorem applies to find the data 
	$$\{x,y, q(z), d\tau_0,d\tau_1,d\tau_2,A(z),B(z)\}.$$ 
	Another application of the theorem could give a different set of data
	$$\{\tilde x,\tilde y, \tilde q(z), d\tilde\tau_0, d\tilde\tau_1, d\tilde\tau_2, \tilde A(z),\tilde B(z)\}$$
	over $U$.
	
	\begin{lemma}\label{lemRatLin}
		$\tilde x$ is a rational-linear function of $x$ and $\tilde y$ is a rational-linear function of $y$.
	\end{lemma}
	\begin{proof}
		From the way that the coordinates are constructed in theorem \ref{thmAmbitoric}, $dx$ and $d\tilde x$ are sections of the line bundle $\xi^*$, so that $\tilde x$ is a function of $x$. Similarly  $\tilde y$ is a function of $y$. From \cite{apostolov2013ambitoric}, we know that $f(x,y)=\frac{q(x,y)}{x-y}$. Seeing how $f$ is represented in these two different coordinates gives
		$$\frac{q(x,y)}{x-y}
		=\frac{\tilde q(\tilde x,\tilde y)}{\tilde x-\tilde y}.$$
		Fix a real number $y_0$. Then we have that 
		$$\frac{q(x,y_0)}{x-y_0}
		=\frac{\tilde q(\tilde x,\tilde y(y_0))}{\tilde x-\tilde y(y_0)}.$$
		The left hand side of this equality is a rational-linear function of $x$, since $q(x,y_0)$ is a linear function of $x$. Similarly, the right hand side is a rational-linear function of $\tilde x$. Thus we can solve the equation to find that $\tilde x$ is a rational-linear function of $x$. A similar argument shows that $\tilde y$ is rational-linear function of $y$.
	\end{proof}
	
	\begin{lemma}\label{lemGaugeEqns}
		$$\frac{(\tilde x-\tilde y)^2}{(x-y)^2}=\frac{d\tilde x}{dx}\frac{d\tilde y}{dy}, \qquad \frac{d\tilde x}{dx} = \frac{\tilde q(\tilde x)}{q(x)},
		\qquad \frac{d\tilde y}{dy} = \frac{\tilde q(\tilde y)}{q(y)}.$$
	\end{lemma}
	\begin{proof}
		One can compute
		\begin{align*}
		df
		&=\frac{dx\big((x-y)(q_0y+q_1)-q(x,y)\big)
		+dy\big((x-y)(q_0x+q_1)+q(x,y)\big)}{(x-y)^2}\\
		&= \frac{-q(y)dx+q(x)dy}{(x-y)^2} 
		= \frac{-\tilde q(\tilde y)d\tilde x+\tilde q(\tilde x)d\tilde y}
		{(\tilde x-\tilde y)^2}
		\end{align*}
		Since $dx\wedge d\tilde x = 0 = dy \wedge d\tilde y$, the above equation can be rearranged to form
		$$\frac{(\tilde x-\tilde y)^2}{(x-y)^2}
		=\frac{d\tilde x}{dx}\frac{\tilde q(\tilde y)}{q(y)}
		=\frac{d\tilde y}{dy}\frac{\tilde q(\tilde x)}{q(x)}.$$
		Note that the terms in this equation are expressible as the product of functions of $x$ and $y$. Thus 
		$$0=\frac{\partial^2}{\partial x\partial y}
		\log\left(\frac{(\tilde x-\tilde y)^2}{(x-y)^2}\right) 
		= \frac2{(\tilde x-\tilde y)^2}\frac{d\tilde x}{dx}\frac{d\tilde y}{dy}
		- \frac2{(x-y)^2}.$$
		The two above equations can be combined to find
		\begin{align*}
		\frac{dx}{q(x)} = \frac{d\tilde x}{\tilde q(\tilde x)}, \qquad
		\frac{dy}{q(y)} = \frac{d\tilde y}{\tilde q(\tilde y)}.
		\end{align*}
	\end{proof}
	\begin{lemma}
		The rational-linear transformations $\tilde x(x)$ and $\tilde y(y)$ are the same, in the sense that if $\tilde x=\frac{ax+b}{cx+d}$, 
		then $\tilde y=\frac{ay+b}{cy+d}$.
	\end{lemma}
	\begin{proof}
		Let $\tilde x=\frac{ax+b}{cx+d}$ and $\tilde y=\frac{a'y+b'}{c'y+d'}$. We compute
		$$\frac{d\tilde x}{dx} = \frac{ad-bc}{(cx+d)^2}, \qquad 
		\frac{d\tilde y}{dy} = \frac{a'd'-b'c'}{(c'y+d')^2}.$$
		From lemma \ref{lemGaugeEqns}, $\frac{(\tilde x-\tilde y)^2}{(x-y)^2}=\frac{d\tilde x}{dx}\frac{d\tilde y}{dy}$. Plugging in the expressions for $\tilde x(x),\tilde y(x)$ and their derivatives yields
		$$\left(\frac{(ax+b)(c'y+d')-(cx+d)(a'y+b')}{(x-y)(cx+d)(c'y+d')}\right)^2
		=\frac{(ad-bc)(a'd'-b'c')}{(cx+d)^2(c'y+d')^2}.$$
		Without loss of generality, we may assume that $ad-bc$ and $a'd'-b'c'$ are $\pm 1$. The above equation then yields
		\begin{align*}
		ad-bc &= a'd'-b'c', \\
		(ax+b)(c'y+d')-(cx+d)(a'y+b') &= \pm (x-y),
		\end{align*}
		which implies
		\begin{align*}
		ac' = a'c, \qquad ad'-a'd=b'c-bc',\qquad bd'=bd'.
		\end{align*}
		These relations imply that 
		$\begin{pmatrix} a' & b' \\ c' & d' \end{pmatrix} 
		= \pm\begin{pmatrix} a & b \\ c & d \end{pmatrix}$,
		so that $\tilde y = \frac{a'y+b'}{c'y+d'}=\frac{ay+b}{cy+d}$. 
	\end{proof}

	The above lemmas allow us to identify the gauge group of coordinate transformations which preserve the regular ambitoric structure on the set of generic torus orbits with $\mathbb P\SL_2\mathbb R$, where an element $\left[\begin{pmatrix}a& b \\ c & d\end{pmatrix}\right]\in\mathbb P\SL_2\mathbb R$ induces the coordinate transformation $\tilde x = \frac{ax+b}{cx+d}, \tilde y = \frac{ay+b}{cy+d}$. Note that the gauge group was already known in \cite{apostolov2013ambitoric} to be $\mathbb P\SL_2\mathbb R$ by construction, but it is not expressed in local coordinates.
	
	\begin{ex}\label{exInfExtension}
		Consider the ambitoric ansatz space $$\mathcal A:=\mathcal A(q,x^4+1,y^4+1,\mathbb T).$$ Consider the change of gauge $\tilde x = -\frac1x, \tilde y=-\frac1y$ which transforms $\mathcal A$ to $$\tilde{\mathcal A}:=\mathcal A(\tilde q,\tilde A(\tilde x),\tilde B(\tilde y)).$$ Since $\frac{\partial}{\partial x}$ and $\frac{\partial}{\partial\tilde x}$ span the same line bundle in $T\mathcal A$ (namely $\xi$), we can write $g_0$ restricted to this line bundle in both gauges as
		$$g_0|_\xi = \frac{dx^2}{A(x)} = \frac{d\tilde x^2}{\tilde{A}(\tilde{x})}.$$
		This allows us to compute
		$$\tilde{A}(\tilde{x}) = A(x)\left(\frac{d\tilde x}{dx}\right)^2=(x^4+1)\left(\frac{-1}{x^2}\right)^2=\tilde x^4+1.$$
		Similarly, one can compute $\tilde B(\tilde y)=\tilde y^4+1$. The image of this gauge transformation is $\{\tilde x\tilde y\neq 0\}\cap\tilde{\mathcal A}$, but our computations show that $\tilde A(\tilde x)$ and $\tilde B(\tilde y)$ extend to smooth positive functions on $\{\tilde x\tilde y=0\}$. Thus the ambitoric structure extends to the points in $\{\tilde x\tilde y=0\}$ where $(\tilde x-\tilde y)\tilde q(\tilde x,\tilde y)$ does not vanish. In terms of the original gauge, we find that $\{x=\infty\}$ and $\{x=-\infty\}$ (respectively $\{y=\infty\}$ and $\{y=\infty\}$) glue together, with the ambitoric structure induced by the inverse gauge transformation from $\{\tilde x\tilde y=0\}$.
	\end{ex}
	
\subsection{Asymptotic analysis of ambitoric ansatz spaces}\label{secCompletions}

For a regular ambitoric $4$-manifold $M$, there is a dense open subset $\mathring M$ where the torus action is free. Since $\mathring M$ is dense, we can study $M\backslash\mathring M$ by studying the asymptotics of $\mathring M$ with respect to some metric which makes $M$ complete.  Theorem \ref{thmAmbitoric} tells us that $\mathring M$ locally looks like an ambitoric ansatz space. Since the asymptotic behaviour is local, understanding the asymptotics of ambitoric ansatz spaces will allow us to understand the asymptotics of $\mathring M$.

Let $U$ be an open subset of an ambitoric ansatz space $\mathcal A:= \mathcal A(q,A,B,\mathbb T)$. We will study the asymptotics of $U$ by considering its Cauchy completion. However, the Cauchy completion depends on the choice of metric, and the ambitoric structure gives a whole conformal class of choices. The metrics that we will study are the barycentric metric $g_0$, the K\"ahler metrics $g_\pm=f^{\mp 1} g_0$, and the metrics $g_p:=\frac{(x-y)q(x,y)}{p(x,y)^2}g_0$ parametrized by a quadratic polynomial $p$ orthogonal to $q$. By theorem \ref{thmDiagRicci}, the metrics $\{g_p\}_{p\perp q}$ are precisely the $\mathfrak t$-invariant metrics in $[g]$ with diagonal Ricci-tensor. Let $$g\in\{g_0,g_\pm\}\cup\{g_p\}_{p\perp q}.$$ We will study the Cauchy completion of $U$ with respect to $g$. This will be denoted by $U^{g}_C$, following the notation used in the appendix. We are interested in the case when $U$ is identified with a subset of a complete orbifold, so the following definition is natural:

\begin{defn}
 $U$ is $g$-\emph{completable} if its $g$-Cauchy completion $U_C^g$ is an orbifold which is Riemannian with respect to a smooth extension of $g$.
\end{defn}

We will focus on the case when $U=\mathcal A$. We will use the appendix to relate $\mathcal A_C^{g}$ to the Cauchy completion with respect to a more convenient metric. In particular, consider the embedding $\mathcal A\hookrightarrow(\mathbb{RP}^1)^2\times\mathbb T$ formed by identifying $\mathbb {RP}^1$ as $\mathbb R\cup\{\infty\}$ for each of the $x$ and $y$ coordinate axes in $\mathcal A$. Let $g_f$ be the flat metric on $(\mathbb{RP}^1)^2\times\mathbb T$. Since
\begin{align*}
	\mathcal A=
	\bigg\{\big(x,y,\vec t\big)\in(\mathbb R)^2\times \mathbb T:
	A(x)>0,B(y)>0,q(x,y)(x-y)\neq 0\bigg\},
\end{align*}

We find that
\begin{align*}
	\mathcal A_C^{g_f}=
	\overline{\left\{x\in\mathbb RP^1:
	A(x)> 0\right\}}
	\times\overline{\left\{y\in\mathbb RP^1:
	B(y)> 0\right\}}\times\mathbb T,
\end{align*}
where $\overline{\cdot}$ is the closure in $\mathbb {RP}^1$ with the flat topology.
We see that $\partial^{g_f}\mathcal A:=\mathcal A^{g_f}_C\backslash\mathcal A$ can be decomposed into components of two types: level sets of $x$ or $y$, or components of the vanishing locus of $(x-y)q(x,y)$. Since the moment map sends level sets of $x$ or $y$ to lines, we call these components \emph{edges}. We call components of the vanishing locus of $(x-y)q(x,y)$ as \emph{folds}. We will call a fold negative if $f=\frac{q(x,y)}{x-y}$ vanishes on it, and positive if $\frac1f$ vanishes on it, analogous to the folding hypersurfaces $Z_\pm$ from section \ref{secMomFold}. A $\mathbb T$-orbit where a pair of edges or folds meet is called a \emph{corner}. A fold which is not an edge or a corner is called a \emph{proper fold}. Since we are treating corners separately, when we refer to an edge or a proper fold we will mean its interior and not include the adjacent corners.

Since $\mathcal A_C^{g_f}$ is compact, by proposition \ref{propCauchyComparisons} in the appendix, every point $p_1$ in $\mathcal A_C^{g}$ can be represented by a point $p_2$ in $\mathcal A_C^{g_f}$, in the sense that there exists a curve in $\mathcal A$ which converges to $p_1$ in $\mathcal A_C^{g}$ and $p_2$ in $\mathcal A_C^{g_f}$. Note that this correspondence between $p_1\in \mathcal A_C^{g}$ and $p_2\in \mathcal A_C^{g_f}$ is not bijective. For example, let's consider the case where $U=\mathcal A$. If $p_2$ lies on $x=\infty$, then $p_2$ may be represented by two curves $c_\pm(t)$, with $\lim_{t\to\infty}x\circ c_\pm(t)=\pm\infty$. Then $c_\pm$ could not both represent the same point in $\mathcal A_C^{g}$, where $x=\infty$ and $x=-\infty$ are not identified. As another example, if $\mathbb T$ does not act freely at $p_1$, then there are multiple choices for $p_2$, since $\mathbb T$ acts freely on $\mathcal A_C^{g_f}$. We will use this correspondence to describe points in $U_C^{g}$ by the coordinates on $\mathcal A_C^{g_f}$, keeping in mind that on $U_C^{g}$ these may not be coordinate functions and may be multi-valued. 

The correspondence between $\mathcal A_C^{g}$ and $\mathcal A_C^{g_f}$ allows us to decompose $\partial^{g}\mathcal A$ into components corresponding to the components of $\partial^{g_f}\mathcal A$, and we will use the same language to refer to them: folds, edges, and corners. Note that there may be components of $\mathcal A_C^{g_f}$ which are not represented in $\mathcal A_C^{g}$. We will refer to these components as \emph{infinitely distant} with respect to $g$. Formally these lie in the Busemann completion (see lemma \ref{lemCauchyInBusemann}), but we will not dwell on this fact. In the case where $g=g_p$, the vanishing locus of $p(x,y)$ forms an additional boundary component, which we will refer to as $P$. The \emph{proper part} of $P$ is the subset of $P$ of all points which are not folds or edges. 

We will now see what we can deduce from the assumption that $\mathcal A$ is $g$-completable, discussing each type of boundary component in turn.

\subsubsection{Folds}
We start with a long lemma. This argument is much more involved than the others in the section.

\begin{lemma}\label{lemFoldComp}
	If $\mathcal A$ is $g$-completable, then every proper fold is infinitely distant. 
\end{lemma}
\begin{proof}
	Let $F$ be a proper fold which is not infinitely distant. The $\mathbb T$-orbits in $F$ must have well-defined $g$-volume. In \cite{apostolov2013ambitoric2}, the $g_+$ volume of the fibre at $(x,y)$ is shown to be proportional to $\frac{A(x)B(y)}{q(x,y)^4}$. The volumes (upto a constant) of such a fibre with respect to each of the candidate metrics for $g$ are given in the following table:
		
	\begin{tabular}[h]{c||c|c|c|c}
		$g$ & $g_+$ & $g_0$ & $g_-$ & $g_p$ \\
		\hline
		$\Vol_{g}\left(\mathbb T_{(x,y)}\right)$ & $\frac{A(x)B(y)}{q(x,y)^4}$ & $\frac{A(x)B(y)}{(x-y)^2q(x,y)^2}$ & $\frac{A(x)B(y)}{(x-y)^4}$ & $\frac{A(x)B(y)}{p(x,y)^4}$
	\end{tabular}

	If $g$ is $g_0$, we find that $\frac{A(x)B(y)}{(x-y)^2q(x,y)^2}$ must be well-defined on $F$. But $F$ is not a corner or an edge, so $A(x)$ and $B(y)$ are positive functions on $F$. Thus we must have that $(x-y)^2q(x,y)^2$ is non-vanishing on $F$. This contradicts $F$ being a fold, so that $g$ cannot be $g_0$. Similar arguments show that if $F$ is positive, then $g\neq g_-$, and if $F$ is negative, then $g\neq g_+$. It follows that if one writes $g=\phi(x,y)g_0$, then $\phi(x,y)$ vanishes on $F$.
	
	Now we will show that the $g$-distance between $\mathbb T$-orbits in $F$ is zero. To see this, let $(x_1,y_1,\vec t_0)$ and $(x_2,y_2,\vec t_0)$ be two different points in $\mathcal A^{g_f}_C\cap F$. Consider the curve $\gamma$ with image in $\mathcal A^{g_f}_C\cap F\cap\{\vec t=\vec t_0\}$ connecting $(x_1,y_1,\vec t_0)$ to $(x_2,y_2,\vec t_0)$. Let $I_x$ and $I_y$ be the ranges of the functions $x$ and $y$ respectively restricted to the image of $\gamma$. Since $F$ is not a corner or an edge, we can find a $\delta>0$ small enough such that $A(x)$ and $B(y)$ are positive functions on the $\delta$-neighbourhoods (with respect to the metrics induced by $g_f$) $I_x^\delta$ and $I_y^\delta$ of $I_x$ and $I_y$ respectively. We will also choose $\delta$ small enough such that neither $I_x^\delta$ nor $I_y^\delta$ are dense subsets of $\mathbb{RP}^1$. Consider the set  $$V:=\{(x,y,\vec t)\in \mathcal A:x\in I_x^\delta, y\in I_y^\delta, \vec t = \vec t_0\}.$$
	Since neither $I_x^\delta$ nor $I_y^\delta$ are dense subsets of $\mathbb{RP}^1$, the metric induced by $g_f$ on $V$ is uniformly equivalent to $dx^2+dy^2$. The metric induced by $g_0$ on $V$ is $\frac{dx^2}{A(x)}+\frac{dy^2}{B(y)}$. Since $A(x)$ and $B(y)$ are both positive on $V$, $g_0$ is uniformly equivalent to $dx^2+dy^2$, and hence $g_f$, on $V$. Thus there exists some $C>0$ such that $g_0<Cg_f$ on $V$.
	
	Let $\epsilon>0$. Since $\phi(x,y)$ vanishes on $F$, we can approximate $\gamma$ with a curve $\gamma_\epsilon$ connecting $(x_1,y_1,\vec t_0)$ to $(x_2,y_2,\vec t_0)$ with the interior of the image contained in $V$  satisfying $\left|\length_{g_f}(\gamma)-\length_{g_f}(\gamma_\epsilon)\right|<\epsilon$ and $\phi(x,y)\circ\gamma^\epsilon<\epsilon$. Since $g=\phi(x,y)g_0$, we find
	\begin{align*}
	d_{g}\left(\mathbb T_{(x_1,y_1)},\mathbb T_{(x_2,y_2)}\right)
	&\leq\length_{g}(\gamma_\epsilon)
	<\epsilon\length_{g_0}(\gamma_\epsilon)\\
	&<C\epsilon\length_{g_f}(\gamma_\epsilon)
	<C\epsilon\left(\length_{g_f}(\gamma)+\epsilon\right).
	\end{align*}
	Taking $\epsilon\to 0$, we find that the $g$-distance between the $\mathbb T$-fibres $\mathbb T_{(x_1,y_1)}$ and $\mathbb T_{(x_2,y_2)}$ is zero. This means that $F$ represents only one $\mathbb T$-orbit in $\mathcal A^{g}_C$. In order for the $g$-volume of this orbit to be well-defined, we must have $\Vol_{g}(\mathbb T_{(x,y)})$ constant. In the case when $g=g_+$, this means that $q(x,y)^4\propto A(x)B(y)$. But it is only possible to write $q(x,y)^4$ as a product of a function of $x$ and a function of $y$ in the parabolic case, where $F$ is an edge. Similarly, we rule out the case where $g=g_-$ and $g=g_p$ unless $p$ is a constant.
	
	If $g=g_p$ for a constant $p$, then $A(x)$ and $B(y)$ are constants, say $A$ and $B$. We will rule out this case by a curvature computation. Since $(\mathcal A,g_-,\omega_-)$ is K\"ahler, we have that the anti-self dual part of the Weyl curvature tensor of $[g_-]$ is given by
	$$W_-\propto s_-(\omega_-\otimes\omega_-)_0^{\sharp_-},$$
	where $(\cdot)_0$ is the projection onto the trace-free part of $\Sym^2(\Lambda^2\mathcal A)$ and $s_-$ is the scalar curvature of $s_-$. Using $g_p=(x-y)^2g_-$, we estimate
	$$\|\Rm_{g_p}\|_{g_p}\geq\|W_-\|_{g_p}=\frac{p^2\|W_-\|_{g_-}}{(x-y)^2}\propto\frac{|s_-|}{(x-y)^2}.$$
	In \cite{apostolov2013ambitoric}, $s_-$ is computed to be
	$$s_- = -\frac{((x - y)^2,A(x))^{(2)}+((x- y)^2,B(y))^{(2)}}{(x-y)q(x,y)},$$
	where for arbitrary functions $f_1(z,w)$ and $f_2(z)$,
	$$(f_1(z,w),f_2(z))^{(2)}=
	f_1^2\frac{\partial}{\partial z}
	\left(f_1\frac{\partial}{\partial z}\left(\frac{f_2}{f_1^2}\right)\right) 
	= f_1\frac{d^2 f_2}{d z^2}
	-3\frac{\partial f_1}{\partial z}\frac{d f_2}{d z}
	+6\frac{\partial^2 f_1}{\partial z^2}f_2.$$ 
	In particular, since $A(x)$ and $B(y)$ are constants, we find that \newline
	$s_-=-2\frac{A+B}{(x-y)q(x,y)},$ so that $\|\Rm_{g_p}\|_{g_p}\gtrsim\frac{A+B}{(x-y)q(x,y)}$ on $\mathcal A$. Since $\|\Rm_{g_p}\|_{g_p}$ must be well-defined at the fold $F$, where $(x-y)q(x,y)$ vanishes, we must have $A+B=0$. This contradicts $A(x)$ and $B(y)$ being positive functions.
\end{proof}

We can use the following lemma to test the infinitely distant criterion on proper folds. 

\begin{lemma}\label{lemAsyLength}
 Let $g$ be a metric on a dense set of $\mathcal A$. Let $F$ be the vanishing locus of some function $\phi$ on $\mathcal A_C^{g_f}$. Assume that there exists some $r\in\mathbb R$ such that each point of $F$ has a neighbourhood where $\|d\phi\|_g\asymp{\phi^r}$. In other words, there exists some $C>0$ such that  $\frac1{C}\phi^r<\|d\phi\|_g< C{\phi^r}$ on the neighbourhood. Then $F$ is infinitely distant with respect to $g$ if and only if $r\geq 1$.
\end{lemma}
\begin{proof}
	Let  $\gamma:[0,1)\to\mathcal A$ be a curve limiting to $F$. Since $\lim_{t\to 1}\phi\circ\gamma(t)=0$, we can find some $T\in[0,1)$ such that $\gamma|_{[T,1)}$ is transverse to the level sets of $\phi$. Thus the function $\phi\circ\gamma|_{[T,1)}:[T,1)\to\mathbb R$ is invertible with image $(0,\epsilon]$, where $\epsilon = \phi\circ\gamma(T)$. Then $\tilde{\gamma}:=\gamma\circ(\phi\circ\gamma)^{-1}:(0,\epsilon]\to\mathcal A$ is a reparametrization of a tail of $\gamma$ which satisfies $\phi\circ\tilde{\gamma}(\tau)=\tau$. Taking the derivative of this relation, we find that $d\phi\circ\dot{\tilde{\gamma}}=1$. We can choose $\epsilon$ small enough such that the $\frac1{C}\phi^r<\|d\phi\|_g< C{\phi^r}$ holds. We compute
	\begin{align}\label{eqnFoldAsymp}
	\length_{g}(\gamma)&\geq\length_{g}(\tilde{\gamma})
	=\int_0^\epsilon\|\dot{\tilde{\gamma}}(\tau)\|_{g}d\tau
	\geq\int_0^\epsilon g\left(
	 \frac{\nabla^g\phi}{\|\nabla^g\phi\|_g},
	 \dot{\tilde{\gamma}}\right)_{\tilde{\gamma}(\tau)}d\tau\\
	&=\int_0^\epsilon\left(\frac{d\phi\circ\dot{\tilde{\gamma}}}{\|d\phi\|_g}\right)_{\tilde \gamma(\tau)}d\tau\asymp\int_0^\epsilon\frac{d\tau}{\phi^r\circ\tilde{\gamma}(\tau)}
	=\int_0^\epsilon\frac{d\tau}{\tau^r}.\nonumber
	\end{align}
	In the case when $r\geq 1$, this integral is infinite, so that the $g$-length of $g$ is infinite. Since $\gamma$ is an arbitrary curve limiting to $F$, we find that $F$ is $g$-infinitely distant as claimed.
	
	Conversely, consider the case when $r<1$. One can use the $g$-gradient flow of the function $\phi$ to construct a curve $\gamma$ limiting to $F$ which satisfies $\dot{{\gamma}}=(\nabla^g\phi)$. Clearly $\gamma$ admits a parametrization by some parameter $\tau$ satisfying $\phi\circ\gamma(\tau)=\tau$. Let $\tilde{\gamma}$ be the restriction of $\gamma$ to a neighbourhood of $\lim_{\tau\to 0}\gamma(\tau)$ such that $\|d\phi\|_g\asymp{\phi^r}$. Without loss of generality, we can assume that $\gamma=\tilde\gamma$. Then (\ref{eqnFoldAsymp}) holds after replacing the inequalities with equalities. Since $r<1$, we find that the $g$-length of $\gamma$ is finite, so that $F$ is not $g$-infinitely distant.
\end{proof}

\begin{lemma}\label{lemProperFoldClose}\label{lemProperFoldDist}
	If $g\in\{g_0,g_\pm\}\cup\{g_p\}_{p\perp q}$, then all proper folds are not infinitely distant with respect to $g$. 
\end{lemma}
\begin{proof}
	First we consider the case of a negative fold $F$. This is the vanishing locus of $f=\frac{q(x,y)}{x-y}$. Up to a change of gauge, we may assume that $q(x,y)$ is not a constant function. This allows us to identify $F$ as the vanishing locus of $q(x,y)$. It is a simple computation with the normal forms to verify that $dq(x,y)$ and $q(x,y)$ do not share any vanishing points. This, combined with the form of $g_0$ from (\ref{eqnAmbiStructure}) allow us to deduce that $\|dq(x,y)\|_{g_0}$ extends to a positive function on $F$. This allows us to apply the previous lemma with $\phi=q(x,y)$ and $r=0$ to deduce that $F$ is not infinitely distant with respect to $g_0$. Since $g_\pm=f^{\mp 1}g_0$, $\|dq(x,y)\|_{g_\pm}={f^{\pm \frac12}}\|dq(x,y)\|_{g_0}$. Since $(x-y)$ does not vanish on $F$, we can then apply the previous lemma with $r=\pm\frac12$ to deduce that $F$ is not infinitely distant with respect to $g_\pm$. 
	
	Similarly, if $p\neq q$, then applying the previous lemma with $r=0$ gives the analogous result for $g_p$. If $p=q$, then $g_p=g_+$, so we've already seen that $F$ is not $g_p$-infinitely distant. The case of positive folds follows similarly by switching the roles of $q(x,y)$ and $(x-y)$.
\end{proof}

\begin{lemma}\label{lemPropPDist}
	The proper part of $P$ is infinitely distant with respect to $g_p$.
\end{lemma}
\begin{proof}
	If $p$ is constant, then $P$ is empty, so there is nothing to show. As we saw in lemma \ref{lemPMomLine}, $dp(x,y)=\frac12 p'(y)dx+\frac12 p'(x) dy$, which only vanishes along $P$ in the case when $<p,p>=0$. In this case, there is some $\gamma\in\mathbb R$ so that $p(z)=(z-\gamma)^2$, and $\{dp(x,y)=0\}\cap P=\{x=y=\gamma\}.$ This is a $\mathbb T$-orbit in $\mathcal A_C^{g_f}$ which is a corner where $P$ meets the fold $\{x=y\}$. Since we are treating corners separately, we find that $dp(x,y)$ is well-defined and non-vanishing on the proper part of $P$. It follows that $\|dp(x,y)\|_{g_0}$ extends to a positive function on the proper part of $P$. Since $\|dp(x,y)\|_{g_p}=p(x,y)\|dp(x,y)\|_{g_0}$, the result follows by applying lemma \ref{lemAsyLength} with $\phi=p(x,y)$ and $r=1$. 
\end{proof}

\subsubsection{Edges}

We start our investigation of the edges by building criteria to test the infinitely distance condition. The following lemma lies at the heart of the argument:

\begin{lemma}
	Let $N$ be a manifold and $T>0$. Consider the Riemannian manifold 
	$$(\hat N,g):=\big((0,T)\times N,a(t)^2dt^2+g_N(t)\big),$$
	where $a(t)$ is a smooth positive function on $(0,T)$ and $\{g_N(t)\}_{t\in(0,T)}$ is a smooth family of metrics on $N$. Then any smooth curve $\gamma:[0,1)\to\hat N$ satisfying $\lim_{\tau\to 1}t\circ\gamma(\tau)=0$ has finite $g$-length if and only if there exist some $\epsilon>0$ such that $\int_0^\epsilon a(t)dt$ converges.
\end{lemma}
\begin{proof}
	First assume that any curve $\gamma:(0,1)\to\tilde M$ with $\lim_{\tau\to 0}t\circ\gamma(\tau)=0$ has infinite length. In particular, for any $p\in N$ and $\epsilon\in(0,T)$, the curve $\gamma(\tau):=(\epsilon\tau,p)$ has infinite length. But this length is $$\int_0^1\sqrt{g(\dot\gamma(\tau),\dot{\gamma}(\tau))}d\tau
	=\int_0^1\epsilon a(\epsilon\tau)d\tau=\int_0^\epsilon a(t)dt,$$
	so that this integral cannot converge as claimed.
	
	Conversely, suppose that any curve $\gamma:(0,1)\to\tilde M$ with $\lim_{\tau\to 0}t\circ\gamma(\tau)=0$ has infinite length. Fix such a gamma. We may reparametrize $\gamma$ so that for some $\epsilon>0$, $t\circ\gamma(\tau)=\tau$ for all $\tau\in(0,\epsilon)$. We then have
	\begin{align*}
	\infty &> \int_0^1\sqrt{g(\dot\gamma(\tau),\dot{\gamma}(\tau))}d\tau \geq
	\int_0^\epsilon\sqrt{g(\dot\gamma(\tau),\dot{\gamma}(\tau))}d\tau \\&\geq  \int_0^\epsilon\sqrt{a(t)^2dt^2(\dot\gamma(\tau),\dot{\gamma}(\tau))}d\tau
	= \int_0^\epsilon a(\tau)d\tau.
	\end{align*}
	This proves our claim.
\end{proof}

We now move on to consider the case $g=g_0$:

\begin{lemma}\label{lemEdgeDist}
	If an edge $E$ is given by $\{x=x_0\}$ (respectively $\{y=y_0\}$), then $E$ is infinitely distant with respect to $g_0$ if and only if there does not exist an $\epsilon>0$ such that $\int_{x_0}^{x_0+\sigma\epsilon}\frac{dx}{\sqrt{A(x)}}$ (respectively $\int_{y_0}^{y_0+\sigma\epsilon}\frac{dy}{\sqrt{B(y)}}$) converges. Here $\sigma\in\{\pm 1\}$ is chosen so that $x_0+\sigma\epsilon$ or $y_0+\sigma\epsilon$ lies in the range of $x$ or $y$ on $\mathcal A$. Also $x_0$ and $y_0$ are assumed to be finite, which can be arranged after a gauge transformation.
\end{lemma}
\begin{proof}
	We will prove the case for $x_0$ with $\sigma=1$. Consider a $\mathbb T$-orbit $\mathcal O$ on $E$. It has coordinates $(x_0,y_0)$ for some $y_0\in\mathbb R$. We can find a neighbourhood $I_y$ of $y_0$ and some $T>0$ such that 
	$(x_0,x_0+T)\times I_y\times\mathbb T\subset\mathcal A$.  
	Let $N=I_y\times\mathbb T$.
	Consider the following coordinate transformation
	$$\phi:M\to(0,T)\times N:(x,y,\vec t)\mapsto (x-x_-,(y,\vec t)).$$
	It is easy to see that $\phi_* g_0$ is of the form $a(t)^2 dt^2+g_N(t)$, where $\phi^*a(t) = \frac1{\sqrt{A(x)}}$. By the previous lemma, we find that $\mathcal O$ is in the $g_0$-Cauchy completion of $M$ if and only if there exists some $\epsilon>0$ such that the integral 
	$$\int_0^\epsilon a(t)dt =\int_{x_-}^{x_-+\epsilon}\frac{dx}{\sqrt{A(x)}}$$ converges. Since $\mathcal O$ was an arbitrary $\mathbb T$-orbit on $E$, this gives the desired result. 
\end{proof}

To understand how the condition of $g$-completable applies to the edges, we will borrow from the work of \cite{apostolov2013ambitoric2}. They define:

\begin{defn}
	An \emph{ambitoric compactification} is a compact connected oriented $4$-orbifold $M$ with an effective action of $2$-torus $\mathbb T$ such that on the (dense) union $\mathring M$ of the free $\mathbb T$-orbits, there is an ambitoci structure $(g_\pm, J_\pm, \mathbb T)$ for which at least one of the K\"ahler metrics extends smoothly to a toric K\"ahler metric on $(M,\mathbb T)$. An ambitoric compactification is \emph{regular} if the ambitoric structure on $\mathring M$ is regular with $(x,y)$-coordinates that are globally defined.
\end{defn}

We will add some definitions:

\begin{defn}
	A connected component $U$ of an ambitoric ansatz space $\mathcal A$ is \emph{box-type} if all of its folds are also edges or corners. It's clear that such a $U$ must be of the form $(x_-,x_+)\times(y_-,y_+)\times\mathbb T\subseteq\mathcal A$, for some intervals $(x_-,x_+),(y_-,y_+)\subset\mathbb R$. 
\end{defn}

\begin{defn}\label{defCompNormal}
	Let $E$ be an edge given by $\{x=x_0\}$ (respectively $\{y=y_0\}$). We say $E$ has a \emph{compatible normal} if $A(x_0)=0$ (respectively $B(y_0)=0$) and the normal vector to $\mu^\pm(E)$ given by $-2\frac{p^{(x_0)}}{A'(x_0)}$ (respectively $2\frac{p^{(y_0)}}{B'(y_0)}$) lies in the lattice $\Lambda\subset\mathfrak t$, where for each $\gamma\in\mathbb R$, $p^{(\gamma)}\in\mathfrak t$ corresponds to the polynomial $p^\gamma(x,y):=\frac12(x-\gamma)q(y,\gamma)+\frac12q(x,\gamma)(y-\gamma)$ under the identification of $\mathfrak t$ with symmetric quadratic polynomials orthogonal to $q(x,y)$. See \cite{apostolov2013ambitoric} for details on this identification.
\end{defn}

Note that the above definition is given for the case when $x_0$ or $y_0$ is finite. In the infinite case, one can reduce to the finite case by performing a gauge transformation (for example $z\mapsto-\frac1z)$.

We can then paraphrase the classification of ambitoric compactifications (proposition 3 in \cite{apostolov2013ambitoric2}):

\begin{prop}\label{propAmbiCompBox}
	A box-type component of an ambitoric ansatz space is the interior of an ambitoric compactification if and only if each of its edges has a compatible normal.
\end{prop}

This will allow us to prove:

\begin{lemma}\label{lemEdgeComp}
	Let $g\in\{g_0,g_\pm\}\cup\{g_p\}_{p\perp q}$. If $\mathcal A$ is $g$-completable, then each edge is either infinitely distant, a fold, part of $P$,  or has compatible normal. 
\end{lemma}
\begin{proof}
	Let $E$ be an edge which is neither a fold nor infinitely distant nor part of $P$. Since $\mathcal A$ is is $g$-completable, $E$ lies in $\mathcal A^{g}_C$ as a set of points with at worst orbifold singularities, and $g$ extends to $E$ as a smooth metric. Let $\phi(x,y)$ be the function such that $g_+=\phi(x,y)g$. Since $E$ is not a fold, $\phi(x,y)$ is a positive function along $E$. This implies that $g_+$ also extends to $E$ as a smooth metric. 
	
	We want to apply the previous proposition here. The proof of that proposition is local in nature, so we can apply it locally. More explicitly, if both $g_+$ and $\omega_+$ extend to $E$, then $E$ looks like a piece of an ambitoric compactification, so we can deduce that $E$ has a compatible normal. We have shown that $g_+$ extends, but not yet $\omega_+$. However, the proof of the previous proposition does not actually require that $\omega_+$ extends, although it certainly is necessary by the converse. To see this, let $p$ be a point on $E$. Without loss of generality, by switching the roles of $x$ and $y$ if necessary, $E$ is of the form $\{x=x_0\}$ and the $\mathbb T$-orbit $\mathbb T\cdot p$ is given by $\{x=x_0,y=y_0\}$. Let $\mathcal A_{y=y_0}=\{(x,y,\vec t)\in\mathcal A:y=y_0\}.$ Then the normal bundle to $\mathbb T\cdot p$ in $T_p\mathcal A_{y=y_0}$ is a vector space which coincides with the space $V_p$ in the proof of proposition $1$ of \cite{apostolov2004hamiltonian}. The construction of $V_p$ was the only part of that direction of that proof which used the symplectic structure, so we can use it to deduce boundary conditions which are shown in \cite{apostolov2013ambitoric2} to be equivalent to the condition that $E$ has a compatible normal.
\end{proof}

\begin{lemma}\label{lemEdgeFoldComp}
	Let $g\in\{g_0,g_\pm\}\cup\{g_p\}_{p\perp q}$. If $\mathcal A$ is $g$-completable, then every edge which is also a fold but neither infinitely distant nor part of $P$ must have a compatible normal and $g\in\{g_-,g_p\}$.
\end{lemma}
\begin{proof}
	Let $E$ be such an edge. An edge can only be a fold in the parabolic case, and the fold must be negative. We will change the gauge so that $q(x,y)=xy$. We consider the case when $E$ is given by $\{x=0\}$, while the case $\{y=0\}$ can be treated similarly. 
	
	To rule out $g\in\{g_0,g_+\}$, we use the fact that the volume of a $\mathbb T$-fibre on $E$ must be well-defined, as in lemma \ref{lemFoldComp}. First consider the case $g=g_0$. We saw in lemma \ref{lemFoldComp} that the $g_0$ volume of the $\mathbb T$-fibre at $(x,y)$ is given by $\frac{A(x)B(y)}{(x-y)^2x^2y^2}$. For this to be defined at $x=0$, $A(x)$ must vanish to at least second order at $0$. By lemma \ref{lemEdgeDist}, this implies that $E$ is infinitely distant with respect to $g_0$, contradicting the definition of $E$.
	
	Similarly for $g=g_+$, we find that $\frac{A(x)B(y)}{x^4y^4}$ must be well-defined, implying that $A(x)$ vanishes to order $4$ at $0$. A similar argument to lemma \ref{lemEdgeDist} shows that if $q(x,y)=xy$, then $E$ is infinitely distant with respect to $g_+$ if and only if there does not exist an $\epsilon>0$ such that the integral $\int_0^\epsilon\frac{xdx}{\sqrt{A(x)}}$ converges. This is the case when $A(x)$ vanishes to order $4$ at $0$, contradicting the definition of $E$ as not infinitely distant with respect to $g_+$.
	
	The fact that $E$e must have a compatible normal follows the same argument as the previous lemma, replacing $g_+$ and $\omega_+$ with $g_-$ and $\omega_-$.
\end{proof}

\begin{lemma}
	If $\mathcal A$ is $g_p$-completable, then every edge which is a part of $P$ is infinitely distant with respect to $g_p$.
\end{lemma}
\begin{proof}
	Let $E$ be an edge which is part of $P$ which is not infinitely distant with respect to $g_p$. We will treat the case where $E$ is given by $\{x=\gamma\}$ for some $\gamma\in\mathbb R$. The case where $E=\{y=\gamma\}$ can be treated similarly, and the case when $\gamma=\infty$ can be reduced to the finite case by a gauge transformation. Since $E$ is part of $P$, there exist some $c\in\mathbb R$ such that $p(x,y)=c(x-\gamma)(y-\gamma)$. Since $\mathcal A$ is $g_p$-completable and $E$ is not infinitely distant, we know that the volume of a $\mathbb T$-fibre on $E$ must be well defined. Again referring to the table from lemma \ref{lemFoldComp}, this volume is $\frac{A(x)B(y)}{p(x,y)^4}$. In order for this to be well-defined, $A(x)$ must vanish to order at least $4$ at $\gamma$.
	
	Now consider the function $\phi=x-\gamma$ on $\mathcal A$. We compute
	$\|d\phi\|_{g_p}=\frac{p(x,y)\sqrt{A(x)}}{\sqrt{q(x,y)(x-y)}}$.
	Since $p(x,y)$ vanishes to order $1$ on $E$ and $A(x)$ vanishes to order at least $4$, while $q(x,y)$ vanishes to order $1$ or $0$ (if $q=p$ or $q\neq p$ respectively), we find that $\|d\phi\|_{g_p}\gtrsim\phi^{\frac32}$. We now apply lemma \ref{lemAsyLength} to deduce the result.
\end{proof}

\subsubsection{Corners}

\begin{lemma}\label{lemCornerInfinite}
	If $\mathcal A$ is $g$-completable, then every corner where an infinitely distant edge or fold meets another edge or fold must be infinitely distant.
\end{lemma}
\begin{proof}
	Assume that two infinitely distant edges or folds meet at a corner $C$ which is not infinitely distant. We can find a small neighbourhood of this corner of the form $$((\mathbb C\backslash \{0\})\times(\mathbb C\backslash \{0\}))\cup C,$$
	where $C$ is glued into $(0,0)$ as a subgroup of $\mathbb T$. This is never an orbifold, contradicting $\mathcal A$ being $g$-completable.
	
	Similarly, let $C'$ be a corner where an infinitely distant edge or fold $E_\infty$ meets an edge or fold $E_f$ which is not infinitely distant. By lemmas \ref{lemFoldComp}, \ref{lemEdgeComp} and \ref{lemEdgeFoldComp}, $E_f$ must be an edge with a compatible normal. It follows that points of $E_f$ have neighbourhoods of the form $\mathbb C/\Gamma\times\mathbb C$, where $\Gamma$ is the orbifold covering group. It follows that $C$ has a neighbourhood of the form $$((\mathbb C/\Gamma)\times(\mathbb C\backslash \{0\}))\cup C,$$ where $C$ is glued into $(0,0)$ as a subgroup of $\mathbb T$. This is never an orbifold, contradicting $\mathcal A$ being $g$-completable.	
\end{proof}

\begin{lemma}\label{lemCornerFold}
	Let $g\in\{g_0,g_\pm\}\cup\{g_p\}_{p\perp q}$. If $\mathcal A$ is $g$-completable and $C$ is a corner at the intersection of two edges which are not infinitely distant, then if $C$ is a positive fold, then $g\in\{g_+,g_p\}$, while if $C$ is a negative fold, then $g\in\{g_-,g_p\}$. If $g=g_p$, then $C$ cannot be part of $P$.
\end{lemma}
\begin{proof}
	We will prove the case that $C$ is a positive fold. The case when $C$ is a negative fold or part of $P$ is treated similarly. By lemmas \ref{lemEdgeComp} and \ref{lemEdgeFoldComp}, the edges adjacent to $C$ must have compatible normals. This implies that $A(x)$ and $B(y)$ both vanish to order $1$ at $C$. For topological reasons, $C$ cannot be a free $\mathbb T$ fibre with respect to $g$. In particular, the volume of the fibre at $C$ must vanish. Again we refer to the table of fibre volumes in lemma \ref{lemFoldComp} to deduce that $g\in\{g_+,g_p\}$. 
\end{proof}

Now we can combine all of our asymptotic analysis to prove our classification of completable ambitoric ansatz spaces:

	\begin{proof}[Proof of theorem \ref{thmClassCompletable}]
		First assume that $\mathcal A$ is $g$ completable. By lemma \ref{lemFoldComp}, every proper fold must be infinitely distant, but by lemma \ref{lemProperFoldDist} no proper fold is infinitely distant. Thus $\mathcal A$ can have no proper folds, which is the first statement. The next two statements are the result of lemmas \ref{lemEdgeComp} and \ref{lemEdgeFoldComp}. 
		
		By lemma \ref{lemCornerFold}, every corner is infinitely distant unless it is at the intersection of two folds or edges which are not infinitely distant. However, we've seen that a fold or an edge which is not infinitely distant must be an edge with a compatible normal. This added to lemma \ref{lemCornerFold} gives the last condition.
		
		Conversely, assume that the stated conditions hold. The first condition combined with lemma \ref{lemPropPDist} allows us to decompose $\mathcal A^g_C$ into edges and corners.
		
		The second condition implies that each edge which is not infinitely distant has a compatible normal. Let $E$ be such an edge. Since the proof of proposition \ref{propAmbiCompBox} is local in nature, it allows us to deduce that for one of the metrics $g_\pm$, $E$ lies in $\mathcal A^{g_\pm}_C$ as a set of points with at worst orbifold singularities, and $g_\pm$ extends to $E$ as a smooth metric. In the case that $E$ is not a fold, then the conformal factor relating $g$ to $g_\pm$ extends to a positive function on $E$, so that $g$ also extends to $E$ as a smooth metric. If $E$ is a fold, then the choice of $g_\pm$ must have been $g_-$ by the third condition. The third condition also ensures that $g = \eta(x,y)g_-$, for some $\eta(x,y)\in\{1,\frac{(x-y)^2}{p(x,y)^2}\}$. Since $\eta(x,y)$ is positive on $E$, $g$ extends to $E$ as a smooth metric.
		
		Finally, let $C$ be a corner which is not infinitely distant. By the last condition, $C$ must lie at the intersection of two edges which are not infinitely distant, and thus have compatible normals. As in the previous paragraph, we can locally apply proposition \ref{propAmbiCompBox} to deduce that for one of $g_\pm$, $C$ lies in $\mathcal A^{g_\pm}_C$ as a point with at worst an orbifold singularity. The conditions that we put on $g$ if $C$ is a fold ensures that the conformal factor relating $g$ to $g_\pm$ on $\mathcal A$ extends to $\mathcal C$ as a positive function, ensuring that $g$ extends smoothly to $C$. 
		
		We've shown that every point in $\mathcal A^{g}_C$ has at worst orbifold singularities, and that $g$ extends to a smooth metric on $\mathcal A^{g}_C$. Thus $\mathcal A$ is $g$-completable as claimed.
		
		To prove the \emph{moreover}, first note that if the ambitoric structure extends to $\mathcal A_C^{g}$, then the function $f$ is a conformal factor between the metrics $g_+$ and $g_0$ on $\mathcal A_C^{g}$. In particular, $f$ must be a smooth positive function on $\mathcal A_C^{g}$. It follows that $\mathcal A_C^{g}$ has no folds. In other words, every fold is infinitely distant.
		
		Conversely, if every fold is infinitely distant, then the functions $q(x,y)$ and $(x-y)$ extend to smooth positive functions on $\mathcal A_C^{g}$. We can use these functions to construct the conformal factor between $g_\pm$ and $g_-$ on $\mathcal A_C^{g}$, so that $g_\pm$ both extend as smooth metrics to $\mathcal A_C^{g}$. Similarly, by considering the expressions for $\omega_\pm$ in (\ref{eqnAmbiStructure}), these also extend to $\mathcal A_C^{g}$, so that $\mathcal A_C^{g}$ is ambitoric.
	\end{proof}

\subsection{Classification of regular ambitoric orbifold completions}

In this section, we will use the results of the previous section to study regular ambitoric orbifold completions (recall that these are defined in the introduction). In the following section, we will extend the local classification of the set of free orbits of regular ambitoric orbifold completion (corollary \ref{corLocClass}) to a local classification on the completion. In the following section, we modify the tool of the Lokal-global-Prinzip for convexity theorems \cite{hilgert1994coadjoint}. Finally, we show how this modified Lokal-global-Prinzip can be applied to get classification results for regular ambitoric orbifold completions.
 
 \subsubsection{Local asymptotics}\label{secLocAsy}
 
 We will first work with local embeddings. Recall corollary \ref{corLocClass}:

\begin{cor}
	Let $(M,[g_0],J_+,J_-,\mathbb T)$ be a regular ambitoric $4$-manifold freely acted on by a $2$-torus $\mathbb T$. Then for any point $p\in M$, there exists a quadratic polynomial $q(z)$, functions $A(z)$ and $B(z)$, a $\mathbb T$-invariant neighbourhood $U$ of $p$, and an ambitoric embedding $\phi:U\hookrightarrow\mathcal A(q,A,B,\mathbb T)$.
\end{cor}

We can use the following lemma to glue these local embeddings together:

\begin{lemma}\label{lemCechEmbedding}
	Let $(M,[g_0],J_+,J_-,\mathbb T)$ be a regular ambitoric $4$-manifold with a free $\mathbb T$-action such that $M/\mathbb T$ is contractible. Then there exists a $\mathbb T$-equivariant local embedding $\phi:M\to(\mathbb{RP}^1)^2\times\mathbb T$ such that for any $p\in M$, there exists an ambitoric embedding  $\phi_p:U_p\hookrightarrow\mathcal A_p$ of a neighbourhood $U_p$ of $p$ into an ambitoric ansatz space $\mathcal A_p$ such that the following diagram commutes:
	\begin{center}
		
	\begin{tikzpicture}[scale=1.5]
	\node (A) at (0,1) {$M$};
	\node (B) at (2,1) {$(\mathbb {RP}^1)^2\times\mathbb T$};
	\node (C) at (0,0) {$U_p$};
	\node (D) at (2,0) {$\mathcal A_p$};

	\path[->,font=\scriptsize,>=angle 90]
	
	(A) edge node[above]{$\phi$} (B)
	(C) edge [right hook->] (A)
	(C) edge [right hook->] node[above]{$\phi_p$} (D)
	(D) edge [right hook->] (B);
	\end{tikzpicture}
	\end{center}
	Moreover, if $\phi$ is injective and for each $z\in\mathbb{RP}^1$, $\phi(M)\cap\left(\{z\}\times\mathbb{RP}^1\times\mathbb T\right)$ and $\phi(M)\cap\left(\mathbb{RP}^1\times\{z\}\times\mathbb T\right)$ are connected or empty, then $\phi$ is an ambitoric embedding into some ambitoric ansatz space $\mathcal A$.
\end{lemma}
\begin{proof}
	From corollary \ref{corLocClass}, $M$ is covered by $\mathbb T$-invariant charts $\{\phi_\alpha:U_\alpha\to\mathcal A_\alpha\}_\alpha$ from $M$ into ambitoric ansatz spaces $\{\mathcal A_\alpha\}_\alpha$.		
	From section \ref{secGauge}, if $U_\alpha\cap U_\beta\neq\emptyset$, then $$\phi_\beta\circ\phi_\alpha^{-1}:\phi_\alpha(U_\alpha\cap U_\beta)\to\phi_\beta(U_\alpha\cap U_\beta)$$ is given by some element of $\mathbb P\SL_2\mathbb R$. Thus the maps $\{\phi_\alpha\}_\alpha$ induce a \u Cech co-cycle $C$ in $H^1(M,\mathbb P\SL_2\mathbb R)$. Since the $\phi_\alpha$ are $\mathbb T$-equivariant, $C$ is $\mathbb T$ invariant. Thus $C$ is equivalent to a \u Cech co-cycle $\bar C$ in $H^1(M/\mathbb T,\mathbb P\SL_2\mathbb R)$. Since $M/\mathbb T$ is contractible, $\bar C$, and hence $C$, must be trivial. This means that the gauge can be chosen globally, so that the $\phi_\alpha s$ can be glued to form a map $\phi:M\to(\mathbb{RP}^1)^2\times\mathbb T.$
	
	Now let $p\in M$. Then $p\in U_\alpha$ for some alpha. Transforming the gauge of $U_\alpha\hookrightarrow\mathcal A_\alpha$ to match the global gauge used to construct $\phi$ results in an embedding $U_p\hookrightarrow\mathcal A_p$ which has the desired properties by construction.
	
	To prove the \emph{moreover}, assume that $\phi$ has the required properties. We need to show that there exist functions $A(x)$, $B(y)$ and $q(x,y)$ on the image of $\phi$ which agree with those given be the local embeddings $\{\phi_p\}_{p\in M}$. For each $x_0\in\mathbb{RP}^1$, since $\phi(M)\cap\left(\{x_0\}\times\mathbb{RP}^1\times\mathbb T\right)$ is connected, the $\{\phi_p\}_{p\in M}$ must agree on the value of $A(x_0)$. Since $x_0$ was arbitrary, it follows that $A(x)$ is uniquely determined on $\phi(M)$. Similarly $B(y)$ is uniquely determined on $\phi(M)$. The fact that $q(x,y)$ is uniquely determined follows from the fact that it is a quadratic function, thus determined on a connected set by its value on any open subset.
\end{proof}

We will have to work a little harder to get asymptotic information from our local ambitoric charts.

Let $(M, [g_0],J_+,J_-,\mathbb T)$ be a regular ambitoric $4$-manifold with free $\mathbb T$-orbits.
Let $\{\lambda_1,\lambda_2\}$ be a set of generators of $\Lambda$. $$\lambda_1^2+\lambda_2^2\in\Sym^2(\mathfrak t)\cong\Sym^2(\mathfrak t^*)^*$$ is a Riemannian metric on $\mathfrak t^*$, identifying $\mathfrak t^*$ with Euclidean space.
The momentum map $\mu^+:M\to\mathfrak t^*$ is $\mathbb T$-invariant, so it descends to an \emph{orbital momentum map} $\bar{\mu}^+:M/\mathbb T\to\mathfrak t^*,$ which is a local embedding. Thus $\bar g_E:=(\bar{\mu}^+)^*(\lambda_1^2+\lambda_2^2)$ is a Riemannian metric on $M/\mathbb T$. Since $\mathfrak t$ is the tangent space to $\mathbb T$, $(\lambda_1^2+\lambda_2^2)^{-1}$ is a Riemannian metric on $\mathbb T$. Thus 
$$g_E:=  ({\mu}^+)^*(\lambda_1^2+\lambda_2^2)+(\lambda_1^2+\lambda_2^2)^{-1}$$
is a $\mathbb T$-invariant metric on $M$ which descends to the metric $\bar g_E$ on $M/\mathbb T$. $g_E$ will be used in the same way at the flat metric $g_f$ was used for ambitoric ansatz spaces.

\begin{lemma}\label{lemLocAmbiBdry}
	Let $(M,\mathring M,g,J_+,J_-,\mathbb T)$ be a regular ambitoric orbifold completion, where $g\in\{g_0,g_\pm\}\cup\{g_p\}_{p\perp q}$. Then every point $a$ in $\partial^{g_E}\mathring M$ has a neighbourhood in $\mathring M_C^{g_E}$ which can be identified with an open set in $\mathcal A^{g_f}_C$ for some ambitoric ansatz space $\mathcal A$ such that $a$ is identified with a point in $\partial^{g_f}\mathcal A$.
\end{lemma}
\begin{proof}
	We first treat the case when $\mathbb T\cdot a$ admits a neighbourhood $U$ in $\mathring M_C^{g_E}/\mathbb T$ such that $U\cap(\mathring M/\mathbb T)$ is contractible. Let $\hat U\subseteq \mathring M$ be the union of the orbits $U\cap(\mathring M/\mathbb T)$. By lemma \ref{lemCechEmbedding}, by shrinking $U$ if necessary, we can find an ambitoric embedding $\phi:\hat U\to\mathcal A$ of $\hat U$ into some ambitoric ansatz space $\mathcal A$. It is easy to see that $\phi^*g_f$ induces the same topology as $g_E$ on $\hat U_C^{g_E}$. Thus we find that $\phi$ naturally extends to an injection 
	$\bar{\phi}:\hat U_C^{g_E}\to\mathcal A_C^{g_f}$. Since $a\in\partial^{g_E}\hat U$, we have $\bar\phi(a)\in\partial^{g_f}\phi(\hat U)$. 
	
	If $\bar\phi(a)\notin\partial^{g_f}\mathcal A$, then $\bar\phi(a)\in\mathcal A$. 
	Since $g|_{\mathring M}$ is one of the special metrics $\{g_0,g_\pm,g_p\}$ of the ambitoric structure on $\mathring M$, $\phi_*\left({g|_{\mathring M}}\right)$ must be homothetic to one of the special metrics of the ambitoric structure on $\phi(\hat U)\subseteq\mathcal A$. Thus $\phi_*\left({g|_{\mathring M}}\right)$ extends to a complete metric on $\mathcal A$, implying that \newline
	$\bar{\phi}(a)\in\partial^{\phi_*\left({g|_{\mathring M}}\right)}\phi(\hat U)$. Moreover, since $\hat U$ is $\mathbb T$-invariant, so is $\phi(\hat U)$, and thus $\partial^{\phi_*\left({g|_{\mathring M}}\right)}\phi(\hat U)$. Thus the free orbit $\mathbb T\cdot\bar\phi(a)$ lies in 
	$\partial^{\phi_*\left({g|_{\mathring M}}\right)}\phi(\hat U)$. Reversing the correspondence shows that there is a free $\mathbb T$-orbit 
	$\mathcal O\subset\partial^{g}\hat U$ corresponding to the orbit 
	$\mathbb T\cdot p\subset\partial^{g_E}\hat U$ via lemma \ref{propCauchyComparisons}. Since $\mathcal O$ is a free orbit, it must lie in $\mathring M$ by definition. But on $\mathring M$, the topologies induced by $g_E$ and $g$ are equivalent. Thus $a\in\mathring M$. This contradicts $a\in\partial^{g_E}\mathring M$, so that $\bar\phi(a)\in\partial^{g_f}\mathcal A$ as claimed.
	
	We now treat the case where every neighbourhood $U$ of $\mathbb T\cdot a$ satisfies the condition that $U\cap(\mathring M/\mathbb T)$ is not contractible. We will show that this case is not possible, essentially because for an ambitoric ansatz space $\mathcal A$, $\partial^{g_f}\mathcal A$ does not have isolated $\mathbb T$-orbits. Since $M_C^{g_E}/\mathbb T$ is $2$-dimensional, we can find a neighbourhood $U$ of $\mathbb T\cdot a$  such that $U\cap(\mathring M/\mathbb T)$ is homeomorphic to a punctured disc. Let $V\subset U\cap(\mathring M/\mathbb T)$ be a contractible set obtained by cutting a line $L$ between the two boundary components of the puctured disc $U\cap(\mathring M/\mathbb T)$. Let $\hat V\subseteq\mathring M$ be the union of $\mathbb T$-orbits $V$. We can then use lemma \ref{lemCechEmbedding} to get a local embedding $\phi:\hat V\to (\mathbb{RP}^1)^2\times\mathbb T$. Arguing as in the previous case, $\phi$ extends to a local embedding 
	$\bar{\phi}:\hat V_C^{g_E}\to(\mathbb{RP}^1)^2\times\mathbb T$. Still using lemma \ref{lemCechEmbedding}, Since $\hat V^{g_E}_C$ is compact, we can cover $\hat V$ with finitely many $\{\hat V_\alpha\}_\alpha$ such that each $\phi|_{\hat V_\alpha}$ is an embedding into some ambitoric ansatz space 
	$\mathcal A_\alpha\subset(\mathbb{RP}^1)^2\times\mathbb T$. Let $S$ be the set of $\alpha$ such that $a\in\partial^{g_E}\hat V_\alpha$.  We can repeat the argument from the previous case for each $\alpha$ to deduce that $\bar{\phi}(a)\in\partial^{g_f}\mathcal A_\alpha$ for each $\alpha\in S$. But the $\{\partial^{g_f}\mathcal A_\alpha\}_\alpha$ are (possibly singular) hypersurfaces which must locally agree, so there must be some neighbourhood $W$ of $\bar\phi(a)$ and some hypersurface $H$ in $W$ passing through $\bar{\phi}(a)$ which such that $\partial^{g_f}\mathcal A_\alpha\cap W = H$. We treat the case where $H$ is not singular at $\bar{\phi}(a)$ for clarity. By shrinking $V$ small enough so that $\bar\phi(V)\subset W$ we find that $\phi(\hat V)$ lies on one side of $H$ in $W$.
	It follows that the ray in $T_{\mathbb T\cdot a}(V_C^{\bar g_E})$ corresponding to $L$ must cover the line $(T_{\bar \phi(a)}H)/\mathfrak t\subset T_{\mathbb T\cdot \bar{\phi}(a)}(\mathbb{RP}^1)^2$ by the map induced by $\bar\phi$. This is impossible, since a linear map cannot send a ray to a line.
\end{proof}

We use the previous lemma to decompose $\partial^{g_E}\mathring M$ into components corresponding to the decomposition of $\partial^{g_f}\mathcal A$ into folds, edges, $P$, and corners. We will use this language to discuss the components of $\partial^{g_E}\mathring M$. As in the previous section, we will use proposition \ref{propCauchyComparisons} to decompose $\partial^{g}\mathring M$ into components, which we will call folds, edges, $P$, and corners likewise. We call components of $\partial^{g_E}\mathring M$ which are not in $\partial^{g}\mathring M$ \emph{infinitely distant}.

The following example is a variation on example 2.8 from \cite{karshon2009non}, and shows that the orbital momentum map is not always a global embedding:

\begin{ex}
	Consider the elliptic type ambitoric ansatz space 
	$$\mathcal A:=\mathcal A\left(xy+1,x^4+1,y^4+1,\mathbb T\right).$$
	Consider the space $\bar{\mathcal A}$ obtained by gluing the edges $\{x=\infty\}$ and $\{x=-\infty\}$ together, as well as the edges $\{y=\infty\}$ and $\{y=-\infty\}$ together. As in example \ref{exInfExtension}, the ambitoric structure from $\mathcal A$ extends to $\bar{\mathcal A}$. Since the functions $A(x)$ and $B(y)$ are positive on $\bar{\mathcal A}$, $\bar{\mathcal A}$ has no edges. It follows that $\bar{\mathcal A}$ has two connected components. Let $\bar{\mathcal A_0}$ be one of the connected components of $\bar{\mathcal A}$.
	
	From section \ref{secMomFold}, we find that the moment map 
	$\mu^+:\bar{\mathcal A_0}\to\mathfrak t^*$ has an image $\Delta_+$ which is the exterior of a conic $\mathcal C_+$. Since $\mathcal A$ is elliptic type, $\mathcal C_{+}$ is an ellipse. Moreover, the $\mu^+$-fibre over any point in $\Delta_+$ is a single $\mathbb T$-orbit in $\bar{\mathcal A_0}$, allowing us to identify 
	$\bar{\mathcal A_0}/\mathbb T\cong\Delta_+$. Since $\Delta_+$ is not simply connected, its universal cover $\pi:\tilde{\Delta}_+\to\Delta_+$ is not injective. Since 
	$\mu^+:\bar{\mathcal A_0}\to\Delta_+$ is a principal $\mathbb T$-bundle, we can use $\pi$ to construct the pull-back bundle $\tilde{\mathcal  A}:=\pi^*\bar{\mathcal A_0}$ over $\tilde\Delta_+$. Since $\pi$ is a covering map, it extends to a covering map from $\tilde {\mathcal A}$ to $\bar{\mathcal A_0}$, which can be used to induce the ambitoric structure from $\bar{\mathcal A_0}$  to $\tilde {\mathcal A}$.
	
	We find that $\tilde{\mathcal A}$ is an ambitoric manifold with orbit space $\tilde{\mathcal A}/\mathbb T\cong\tilde{\Delta}_+$ and orbital moment map $\pi:\tilde\Delta_+\to\Delta_+\subset\mathfrak t^*$ which is not injective.
\end{ex}

The main feature of the previous example was the existence of a proper fold. This motivates us to consider only regular ambitoric orbifold completions without proper folds, since we want to rule out the possibility that the orbital moment map is not injective. If $(M,\mathring M,g,J_+,J_-,\mathbb T)$ is a regular ambitoric orbifold completion without proper folds, then the boundary $\partial^{g_E}\mathring M$ consists of only edges and the corners where they meet. The orbital moment map naturally extends to a map $(\mathring M/\mathbb T)_C^{\bar g_E}\to\mathfrak t^*$, which is locally convex since $\partial^{g_E}\mathring M$ consists of only edges and corners, which each get locally sent to lines and convex cones respectively. As in the case of compact toric orbifolds, we wish to use this local convexity to prove global convexity, which will give us control over the topology of $M$. The tool to do this is the Lokal-global-Prinzip for convexity theorems \cite{hilgert1994coadjoint}, although it requires that the momentum map is proper. Since $(\mathring M/\mathbb T)_C^{\bar g_E}$ may not be compact a priori, we have to consider the case that the momentum map is not proper. However, in the next section we will see that completeness is actually enough for the proof of Lokal-global-Prinzip to hold.

\subsubsection{Lokal-global-Prinzip}

We begin this section by recalling some definitions from \cite{hilgert1994coadjoint}, and then stating their version of the Lokal-global-Prinzip. We will then state and prove the slight modification of this result which we will need. Note that it is likely that what we prove here is only a case of the general treatment of the Lokal-global-Prinzip in \cite{rump2012convexity} (where we got the idea to use Hopf-Rinow), but for our purposes we find it convenient to prove the result.

\begin{defn}
	A continuous map $\Psi:X\to V$ from a connected Hausdorff topological space $X$ to a finite dimensional vector space $V$ is \emph{locally fibre connected} if every point in $X$ admits an arbitrarily small neighbourhood $U$ such that $\Psi^{-1}(\Psi(u))\cap U$ is connected for each $u\in U$.
\end{defn}

\begin{defn}
	A locally fibre connected map $\Psi:X\to V$ has \emph{local convexity data} if every $x\in X$ admits an arbitrarily small neighbourhood $U_x$ and a closed convex cone $C_x\subseteq V$ with vertex $\Psi(x)$ satisfying:
	\begin{itemize}
		\item $\Psi(U_x)$ is a neighbourhood of $\Psi(x)$ in $C_x$,
		\item $\Psi|_{U_x}:U_x\to C_x$ is open,
		\item $\Psi^{-1}(\Psi(u))\cap U_x$ is connected for each $u\in U_x$,
	\end{itemize}
	where the topology on $C_x$ is the subspace topology induced from $V$.
\end{defn}

These definitions allow us to state the theorem:

\begin{theorem}[Lokal-global-Prinzip for convexity theorems \cite{hilgert1994coadjoint}]
	Let $\Psi: X\to V$ be a proper locally fibre connected map with local convexity data. Then $\Psi(X)$ is a closed locally polyhedral convex subset of $V$, the fibres $\Psi^{-1}(v)$ are all connected, and $\Psi:X\to\Psi(X)$ is a open mapping.
\end{theorem}

Since we are essentially reproducing this theorem, we need to continue defining the constructions used in the proof.

Consider the equivalence relation on $X$ defined by $x\sim y$ if and only if $x$ and $y$ are both contained in the same connected component of a fibre of $\Psi$. The quotient space $\tilde X:=X/\sim$ is called the \emph{$\Psi$-quotient} of $X$. Let $\tilde{\Psi}:\tilde X\to V$ be the map induced by $\Psi$ on $\tilde X$. 

Let $d_V$ be the usual Euclidean distance on $V$ with respect to some fixed basis. We can use this to define a distance on $\tilde X$ as follows:
$$d:\tilde X\times\tilde X\to [0,\infty):(\tilde x,\tilde y)\mapsto \inf_{\gamma\in\Gamma(\tilde x,\tilde y)}\length_{d_V}(\tilde\Psi\circ\gamma),$$
where $\Gamma(\tilde x,\tilde y)$ is the set of curves $\gamma$ connecting $\tilde x$ to $\tilde y$ such that $\Psi\circ\gamma$ is piecewise differentiable. $d$ is called the \emph{metric induced on $\tilde X$} by $\tilde{\Psi}$.

Now we can prove our version of the Lokal-global Prinzip for convexity theorems. Note that the proof is the same as the relevant parts of the previous theorem aside from the replacement of the use of properness with an appeal to the Hopf-Rinow theorem.

\begin{theorem}
	Let $\Psi: X\to V$ be a locally fibre connected map with local convexity data. Moreover, assume that $\tilde X$ is a complete locally compact length space with respect to the metric induced by $\tilde\Psi$. Then $\phi(X)$ is a convex subset of $V$ and the fibres $\Psi^{-1}(v)$ are all connected.
\end{theorem}
\begin{proof}
 Let $\tilde x_0,\tilde x_1\in\tilde X$ with $c:=d(\tilde x_0,\tilde x_1)$. For each $n\in\mathbb N$, we can find some $\gamma_n\in\Gamma(\tilde x_0,\tilde x_1)$ such that $\length_{d_V}(\tilde\Psi\circ\gamma)\leq c+\frac1n$. For each $n\in\mathbb N$, let $\tilde x_{\frac12}^{(n)}$ be the midpoint of $\gamma_n$. Then $\left(\tilde x_{\frac12}^{(n)}\right)_{n=1}^\infty$ is a sequence in the closed ball $B_{c+1}(\tilde x_0)$ of radius $c+1$ about $\tilde x_0$. Since $\tilde X$ is a complete locally compact length space, the Hopf-Rinow theorem \cite{bridson1999metric} tells us that $B_{c+1}(\tilde x_0)$ is compact. Thus  $\left(\tilde x_{\frac12}^{(n)}\right)_{n=1}^\infty$ admits a subsequence which converges to some $\tilde x_{\frac 12}\in B_{c+1}(\tilde x_0)$, which must satisfy
 $$d\left(\tilde x_0,\tilde{x}_{\frac12}\right)
  =d\left(\tilde x_{\frac12},\tilde{x}_1\right)=\frac c2.$$
 
 Repeating the argument for the pairs of points $(\tilde x_0,\tilde x_{\frac12})$ and $(\tilde x_{\frac12},\tilde x_1)$, we construct points $\tilde x_{\frac14}$ and $\tilde x_{\frac34}$ satisfying 
 \begin{equation}\label{eqnLGPLine}
  d\left(\tilde x_{\frac n{2^m}},\tilde x_{\frac {p}{2^{q}}}\right) 
  = c\left|\frac n{2^m}-\frac{p}{2^{q}}\right| 
 \end{equation}
 for all suitable $n,m,p,q$. Inductively we can repeat the argument to construct points $\tilde x_{\frac n{2^m}}$ for $n,m\in\mathbb N$ with $0\leq n\leq 2^m$ satisfying (\ref{eqnLGPLine}). We can then extend the function $\frac{n}{2^m}\mapsto\tilde x_{\frac n{2^m}}$ to a continuous function $\gamma:[0,1]\to\tilde X$ satisfying $d\big(\gamma(t),\gamma(t')\big)=c|t-t'|$ for all $t,t'\in[0,1]$. This means that locally $$d_V\big(\tilde\Psi\circ\gamma(t),\tilde\Psi\circ\gamma(t')\big)=c|t-t'|,$$ which can only happen if $\tilde\Psi\circ\gamma$ is a straight line segment. This line segment connects $\tilde\Psi(\tilde x_0)$ and $\tilde\Psi(\tilde x_1)$. Varying $\tilde x_0$ and $\tilde x_1$, we find that $\Psi(X)=\tilde{\Psi}(\tilde X)$ is convex.
 
 To prove that the fibres are all connected, let $\tilde x_0,\tilde x_1\in\tilde X$ such that $\tilde\Psi(\tilde x_0)=\tilde\Psi(\tilde x_1)=:v$. Let $\gamma:[0,1]\to\tilde X$ be the curve constructed above. Then $\tilde\Psi\circ\gamma$ is a line segment containing $v$. Assume that $\tilde\Psi\circ\gamma$ is not the constant function $v$. Then $\tilde\Psi\circ\gamma$ is a loop through $v$ as well as a straight line segment. Thus there exists a turning point $v_0=\tilde\Psi\circ\gamma(t_0)$ such that $\tilde\Psi\circ\gamma\big([t_0-\epsilon,t_0]\big)=\tilde\Psi\circ\gamma\big([t_0,t_0+\epsilon]\big)$ for all $\epsilon>0$ sufficiently small. Thus $\tilde{\Psi}$ is not locally injective near $\gamma(t_0)$. This contradicts $\Psi$ being locally fibre connected. Thus $\tilde\Psi\circ\gamma$ is the constant function $v$. This implies that $d(\tilde x_0,\tilde x_1)=0,$ so that $\tilde x_0=\tilde x_1$ as required.
\end{proof}

We recall proposition $3.25$ from \cite{bridson1999metric}:

\begin{prop}
	Let $X$ be a length space and $\tilde X$ be a Hausdorff topological space. Let $p:\tilde X\to X$ be a continuous local homeomorphism, and $d$ be the metric induced on $\tilde X$ by p.
	\begin{enumerate}
		\item If one endows $\tilde X$ with the metric $d$, then $p$ becomes a local isometry.
		\item $d$ is a length metric.
		\item $d$ is the unique metric on $\tilde X$ that satisfies properties $(1)$ and $(2)$.
	\end{enumerate}
\end{prop}

Combining this proposition with the previous theorem gives:

\begin{cor}\label{corLGPCompleteCoverish}
	Let $\Psi: X \to V$ be a locally fibre connected map with local convexity data. Moreover, assume that $\tilde\Psi$ is a local homeomorphism, and that $\tilde X$ is complete with respect to the metric $\tilde d$ given by the pull-back with respect to $\tilde \Psi$ of the Euclidean metric on $V$. Then $\phi(X)$ is a convex subset of $V$ and the fibres $\Psi^{-1}(v)$ are all connected.
\end{cor}
\begin{proof}
	By the previous theorem, it suffices to show that $(\tilde X,d)$ is a complete locally compact length space. Since $\tilde \Psi:\tilde X\to V$ is a local homeomorphism and $V$ is locally compact (being a vector space), $\tilde X$ is locally compact. By the previous proposition, $(\tilde X,d)$ is a length space. Note that $\tilde d$ and $d$ both satisfy properties $(1)$ and $(2)$ in the previous proposition. Thus by uniqueness, we have that $d=\tilde d$. Since $(\tilde X,\tilde d)$ is complete, $(\tilde X,d)$ is complete.
\end{proof}

\subsubsection{Classification results}

We can now combine the results of the previous two sections to obtain the main classification results. It is clear from the definitions that the Cauchy completion of a completable ambitoric ansatz space is a regular ambitoric orbifold completion. The following theorem gives a converse in the case that there are no proper folds. Note that this can be combined with our classification of completable ambitoric ansatz spaces to give a classification of regular ambitoric orbifold completion without folds.

The following lemma will be the key which allows us to use our Lokal-global-Prinzip:

\begin{lemma}
 Let $\mathcal A_0$ be a connected component of some ambitoric ansatz space without proper folds. Then the orbital moment map $\bar{\mu}^+:\mathcal A_0/\mathbb T\to\mathfrak t^*$ extends to $(\mathcal A_0/\mathbb T)_C^{\bar g_E}$ as a local homeomorphism with local convexity data.
\end{lemma}
\begin{proof}	
	$\bar{\mu}^+$ is a rational map from $\mathcal A_0/\mathbb T\subset(\mathbb{RP}^1)^2$ to $\mathfrak t^*$. Thus $\bar{\mu}^+$ extends naturally to a map on $(\mathbb{RP}^1)^2$, and hence on $(\mathcal A_0/\mathbb T)_C^{\bar g_E}\subset(\mathbb{RP}^1)^2$, which we will also denote by $\bar{\mu}^+$. Since $\mathcal A_0/\mathbb T$ is connected and $\bar{\mu}^+|_{\mathcal A_0/\mathbb T}$ is a homeomorphism, so is $\bar{\mu}^+|_{(\mathcal A_0/\mathbb T)_C^{\bar g_E}}.$
	
	Since there are no proper folds, $\partial^{\bar g_E}(\mathcal A_0/\mathbb T)$ consists of finitely many edges and components of $P$, which $\bar{\mu}^{+}$ sends to lines tangent to the conic $\mathcal C_+$. We find that for every corner $c\in\partial^{\bar g_E}(\mathcal A_0/\mathbb T)$, arbitrarily small neighbourhoods of $c$ get mapped by $\bar{\mu}^+$ to neighbourhoods of $\bar{\mu}^+(c)$ in the convex cone with vertex $\bar{\mu}^+(c)$ formed by the rays passing through $\mathcal C_+$. By similarly describing neighbourhoods of points on edges, on $P$, and in $\mathcal A_0$, we find that $\bar{\mu}^+:(\mathcal A_0/\mathbb T)^{\bar g_E}_C\to \mathfrak t^*$ has local convexity data.
\end{proof}

We can now prove our classification of regular ambitoric orbifold completions:

\begin{proof}[Proof of theorem \ref{thmCompletionClassification}]
  Since $\mathring M$ is $g$-completable, it has no proper folds by theorem \ref{thmClassCompletable}.
  By lemma \ref{lemLocAmbiBdry}, $(\mathring M/\mathbb T)^{\bar g_E}_C$ is locally covered by charts with image in $(\mathcal A/\mathbb T)_C^{g_f}$, which likewise cannot have proper folds.  Applying the previous lemma to these local charts, we find that $\bar{\mu}^+$ naturally extends to a local homeomorphism $\bar{\mu}^+:(\mathring M/\mathbb T)_C^{\bar g_E}\to\mathfrak t^*$ with local convexity data. 
 
 Since $\bar g_E$ is the pull-back of a Euclidean metric on $\mathfrak t^*$ by $\bar{\mu}^+$, we can apply corollary \ref{corLGPCompleteCoverish} with $X=\tilde X = (\mathring M/\mathbb T)_C^{\bar g_E}$, $V=\mathfrak t^*$ and $\Psi=\tilde \Psi = \bar{\mu}^+$. This tells us that $\bar{\mu}^+$ is injective with convex image. In particular $\mathring M/\mathbb T$ is contractible. We can then apply lemma \ref{lemCechEmbedding} to construct a $\mathbb T$-equivariant local embedding $\phi:\mathring M\to (\mathbb {RP}^1)^2\times \mathbb T$. Let $\bar{\phi}:\mathring M/\mathbb T\to(\mathbb{RP}^1)^2$ be the map induced by $\phi$. We have the following commutative diagram:

 	\begin{center}
 		
 		\begin{tikzpicture}[scale=1.5]
 		\node (A) at (0,1) {$\left(\mathbb {RP}^1\right)^2\times\mathbb T$};
 		\node (B) at (2,1) {$\left(\mathbb {RP}^1\right)^2$};
 		\node (C) at (0,0) {$\mathring M$};
 		\node (D) at (2,0) {$\mathring M/\mathbb T$};
 		\node (E) at (3,1) {$\mathfrak t^*$};
 		
 		\path[->,font=\scriptsize,>=angle 90]		
 		 (C) edge node[left]{$\phi$} (A)
 		 (D) edge node[left]{$\bar{\phi}$} (B)
 		 (A) edge (B)
 		 (C) edge (D)
 		 (B) edge node[above]{$\bar{\mu}^+$} (E)
 		 (D) edge node[below]{$\bar{\mu}^+$} (E)
 		;
 		\end{tikzpicture}
 	\end{center}
 Since $\bar{\mu}^+:\mathring M/\mathbb T\to\mathfrak t^*$ is injective, $\bar\phi$ must be injective. Thus $\phi$ is injective. It also follows that since $\mu^+(\mathring M)$ is convex, $\phi(\mathring M)$ intersects each level set of $x$ or $y$ in at most one connected component. We are now able to apply the full force of lemma \ref{lemCechEmbedding} to deduce that the image of $\phi$ can be given the structure of an ambitoric ansatz space. Since $\mathring M$ has no proper folds, its boundary components are all either edges or parts of $P$. Since the moment map sends each of these boundary components to edges, the moment map of $M$ must be a polygon.
\end{proof}

We are now in a position to prove our classification of regular ambitoric $4$-orbfolds:
\begin{proof}[Proof of corollary \ref{corAmbiClass}]
 A regular ambitoric $4$-orbifold $M$ which is complete with respect to $g_0$ is a regular ambitoric orbifold completion with respect to $g$. We can then apply theorem \ref{thmCompletionClassification} to embed the set of free orbits $\mathring M$ of $M$ in the Cauchy completion (with respect $g$) of an ambitoric ansatz space $\mathcal A(q,A,B,\mathbb T)$ which has polygonal moment map images. Moreover, since $\mathcal A$ is completable, it has no proper folds by theorem \ref{thmClassCompletable}. Since $g\in\{g_0,g_\pm\}$, there is no $P$ boundary component. Combining these facts, we find that $\mathcal A$ is of box-type. In particular there exist $x_\pm,y_\pm\in\mathbb R$ such that $\mathring M\cong (x_-,x_+)\times (y_-,y_+)\times\mathbb T\subset\mathcal A(q,A,B,\mathbb T)$, and that $(x-y)q(x,y)$ does not vanish on $(x_-,x_+)\times (y_-,y_+)$.
 
 By theorem \ref{thmClassCompletable}, each edge is either infinitely distant or has compatible normal, and each fold is infinitely distant. The condition on the edges gives the constraints on $A(x)$ and $B(y)$, using lemma \ref{lemEdgeDist}. By an argument similar to lemma \ref{lemEdgeDist}, each corner is infinitely distant with respect to $g_0$ if and only if either of the adjacent edges is. This implies that if two edges which are not infinitely distant meet at a corner, then that corner cannot be a fold. This gives the last condition.
 
 Conversely, given the data described in corollary, on can construct a torus $\mathbb T=\mathbb R^2/\Lambda$, and an ambitoric ansatz space $\mathcal A(q,A,B,\mathbb T)$. We find that $\mathcal A_0:(x_-,x_+)\times(y_-,y_+)\times\mathbb T\subset \mathcal A(q,A,B,\mathbb T)$ is a connected component of an ambitoric ansatz space with no proper folds and each of its four edges being either infinitely distant or having compatible normals. Then theorem \ref{thmClassCompletable} tells us that $(\mathcal A_0)_C^{g_0}$ is a regular ambitoric $4$-orbifold which is complete with respect to $g_0$.
\end{proof}

\section{Application: Compact $4$-manifolds admitting Killing $2$-forms}\label{secKilling}

In this section, we apply our results to the setting of $4$-manifolds admitting Killing $2$-forms. This will build on the work in \cite{gauduchon2015killing}, which we will summarize now.

\begin{defn}
  A differential form $\psi$ on a Riemannian manifold $(M,g)$ is \emph{conformally Killing} if its covariant derivative $\nabla\psi$ is of the form 
 $$\nabla_X\psi=\alpha\wedge X^\flat + X\lrcorner\beta,\qquad\forall X\in \Gamma(TM),$$
 for some differential forms $\alpha$ and $\beta$. Moreover, $\psi$ is \emph{Killing} if $\alpha=0$ and \emph{*-Killing} if $\beta=0$.
\end{defn}

Note that if $M$ is oriented and $*$ is the Hodge star operator, then $\psi$ is Killing if and only if $*\psi$ is $*$-Killing.

\begin{prop}[2.1 in \cite{gauduchon2015killing}]\label{propGMAmbi}
Let $(M,g)$ be a connected oriented Riemannian $4$-manifold admitting a non-parallel $*$-Killing $2$-form $\psi$. Then on an open dense subset $M_0$ of $M$, the pair $(g,\psi)$ gives rise to an ambiK\"ahler structure $(g_+,J+,\omega_+),(g_-,J_-,\omega_-)$, where $g_\pm=f_\pm^{-2}g$. Here $f_\pm:=\frac{|\psi_\pm|}{\sqrt 2}$, where $\psi_\pm$ are the self-dual and anti-self dual parts of $\psi$.
\end{prop}

Define vector fields $K_1$ and $K_2$ on such an $M$ by
\begin{align*}
 K_1 &:=-\frac12\alpha^\sharp, \\
 K_2 &:=\frac12(*\psi)^\sharp(K_1).
\end{align*}
These can be used to form a rough classification:
\begin{prop}[3.3 in \cite{gauduchon2015killing}]
 Any connected oriented $4$-dimensional Riemannian manifold $(M,g)$ admitting a non-parallel $*$-Killing $2$-from $\psi$ fits into one of the following three exclusive possible cases:
 \begin{enumerate}
 	\item The vector fields $K_1$ and $K_2$ are Killing and independent on a dense open set $\mathcal U$ of $M$.
 	\item The vector fields $K_1$ and $K_2$ are Killing and $K_2=cK_1$ for some $c\in\mathbb R\backslash 0$. 
 	\item $f_+=f_-$ on all of $M$, $K_2=0$, and $K_1$ is not a Killing vector field in general.
 \end{enumerate}
\end{prop}
The authors go on to show that in the first case, $\mathcal U$ admits a regular ambitoric structure $(\mathcal U, [g],J_+,J_-)$ of hyperbolic type, where $g$ is identified with the metric $g_1$ ($g_p$ when $p=1$).
In the language that we have developed, \newline$(M,\mathcal U, g=g_1,J_+,J_-)$ is a regular ambitoric manifold completion. We can then apply our classification result to obtain our claimed result:

\begin{proof}[Proof of theorem \ref{thmIntroKilling}]
	By theorem \ref{thmCompletionClassification},
	 $\mathcal U$ can be embedded into some ambitoric ansatz space $\mathcal A$ whose momentum map images are polygonal. Since \newline$P=\{1=p(x,y)=0\}=\emptyset$, we can argue as in the proof of corollary \ref{corAmbiClass} that $\mathcal U$ is embedded as a box-type subset of $\mathcal A$. It follows that the action of $2$-torus on $\mathcal U$ extends to $M$ with at most $4$ fixed points, corresponding to the corners of the box. We can then apply the work of \cite{orlik1970actions}, which classifies oriented $4$-manifolds equipped with a $2$-torus action. Since the number of fixed points $t$ is at most $4$, this classification tells us that $M$ is (upto orientation reversal) diffeomorphic to one of $\{\mathbb S^4, \mathbb{CP}^2, \mathbb{CP}^2\#\mathbb{CP}^2,\mathbb S^2\times\mathbb S^2,\mathbb{CP}^2\#\overline{\mathbb{CP}^2}\}$.
	
	 Each Hirzebruch surface is diffeomorphic to either $\mathbb S^2\times\mathbb S^2$ or $\mathbb{CP}^2\#\overline{\mathbb{CP}^2}$, so it suffices to rule out $\mathbb{CP}^2\#\mathbb{CP}^2$. To see this, we note that the convexity of $\mu_\pm(\mathcal U)$ implies that we can ensure that the $\{\epsilon_i\}_{i=1}^{t}$ used in \cite{orlik1970actions} are all equal to $1$ by choosing the generators of the isotropy subgroups to correspond to inward normals for the sides of $\mu_\pm(\mathcal U)$. This rules out the case of $\mathbb{CP}^2\#\mathbb{CP}^2$, since this case has $t=4$ and $\epsilon_1\epsilon_4=-\epsilon_2\epsilon_3$.

Conversely, we must show that $\mathbb S^4$, $\mathbb{CP}^2$ and the Hirzebruch surfaces all admit metrics which admit non-parallel $*$-Killing $2$-forms. The case of $\mathbb S^4$ with the usual constant scalar curvature metric is explored in detail in \cite{gauduchon2015killing}. For $\mathbb{CP}^2$ and the Hirzebruch surfaces, we will show that they are regular ambitoric orbifold completions of hyperbolic type. The converse of proposition \ref{propGMAmbi} \cite{gauduchon2015killing} then provides the claimed $*$-Killing $2$-form. We will use the hyperbolic normal form from section \ref{SecHypNorm}, so that $q(x,y)=x+y$ and the conic is given by $\mathcal C_\pm = \{4\mu_1^\pm\mu_2^\pm=-1\}$. We will take $\Lambda$ to be the standard basis for $\mathfrak t$ in the coordinates given in the normal form.

We can construct lines tangent to $\mathcal C_\pm$ with normals 
 $\begin{pmatrix} 1 \\ 0\end{pmatrix},
  \begin{pmatrix} 0 \\ -1\end{pmatrix},$ and
 $\begin{pmatrix} -1 \\ 1\end{pmatrix}$, as in figure \ref{figHypProjPlane}. Each pair of normals generates $\Lambda$, so that the moment map image is a Delzant triangle. It is well known that the Delzant construction applied to a Delzant triangle is $\mathbb {CP}^2$.
 
 \begin{figure}[h]\label{figHypProjPlane}
 	\centering
 	\includegraphics[scale = 0.9]{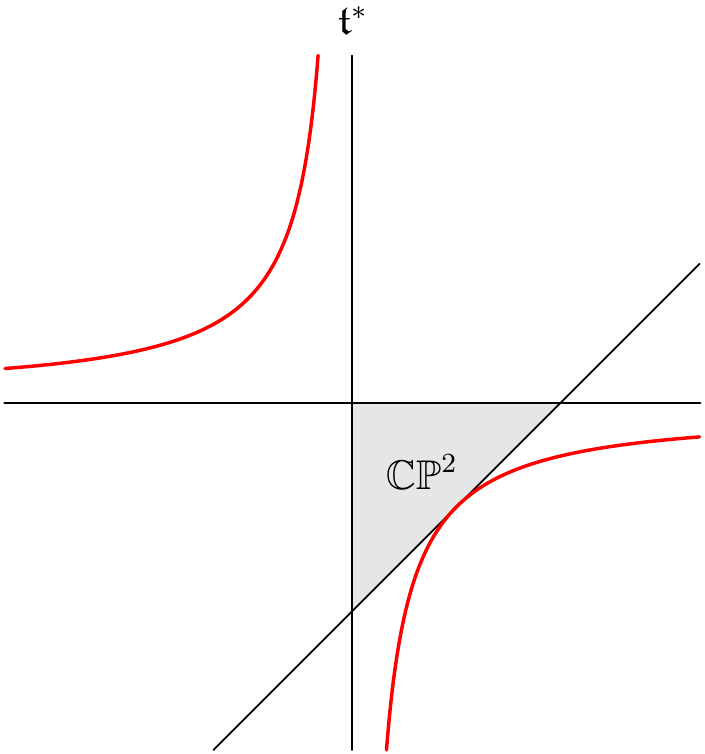}
 	\caption{The moment map image of $\mathbb{CP}^2$ with a hyperbolic ambitoric structure.}
 \end{figure}
 
 We can construct lines tangent to $\mathcal C_\pm$ with normals 
 $\begin{pmatrix} 1 \\ 0\end{pmatrix},
 \begin{pmatrix} -k-1 \\ k\end{pmatrix},
 \begin{pmatrix} 1 \\ -1\end{pmatrix},$ and
 $\begin{pmatrix} -1 \\ 1\end{pmatrix},$ as in figure \ref{figHypHirz}. As in the triangle case, it is easy to check that the trapezoid $\Delta$ formed by these lines is Delzant. Moreover the equation
  $\begin{pmatrix} 1 \\ 0\end{pmatrix} + \begin{pmatrix} -k-1 \\ k\end{pmatrix} = k \begin{pmatrix} -1 \\ 1\end{pmatrix}$ 
  implies that $\Delta$ is the moment-map image of the $k$th Hirzebruch surface $H_k$ (see \cite{karshon2007compact}). Note that to get all of the diffeomorphism types, it suffices to only construct $H_1$ and $H_2$, since the odd (respectively even) order Hirzebruch surfaces are all diffeomorphic (see for example \cite{karshon2007compact}). Moreover, $\mathbb S^2\times\mathbb S^2$ is diffeomorphic to $H_2$, so it is covered by this construction as well.

\begin{figure}[H]\label{figHypHirz}
	\centering
	\includegraphics[scale = 0.9]{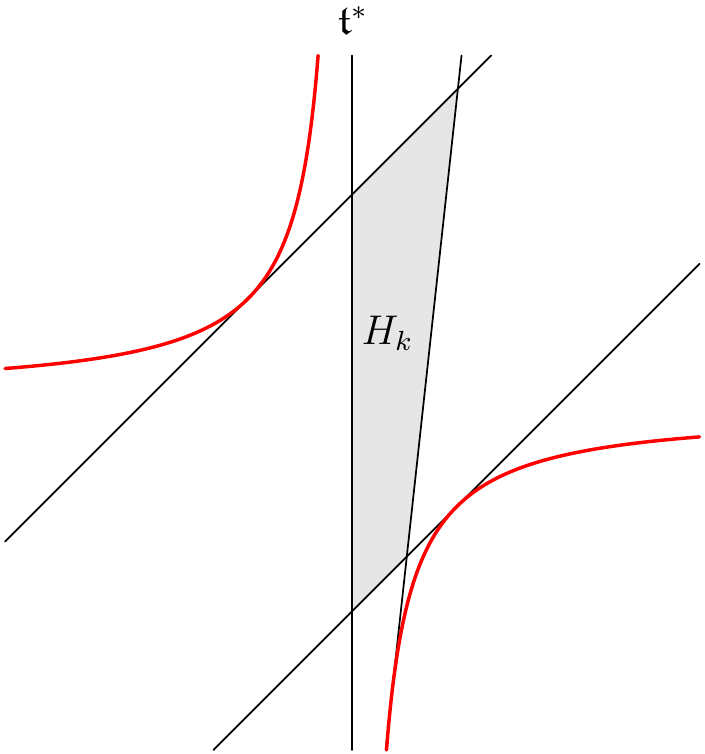}
	\caption{The moment map image of the $k$th Hirzebruch surface $H_k$ with a hyperbolic ambitoric structure.}
\end{figure}

\end{proof}

\section{Acknowledgments}

I would like to thank my doctoral supervisors Niky Kamran and Vestislav Apostolov for suggesting this project for me, as well as providing guidance in both the material and the style. I would also like to thank Yael Karshon for pointing me in the direction of folded symplectic structures.

\appendix

\section{Appendix: Busemann Completions}
 The Busemann completion of a Riemannian manifold is a set which includes the Cauchy completion as well as elements corresponding to directions at infinity. Our goal is to relate the Busemann completions of different metrics on the same manifold. We follow the work of \cite{flores2013gromov}, although we make some cosmetic changes of definitions to suit our goals.
 \begin{defn}
  Let $(M,g)$ be a Riemannian manifold. Let
  $$C(M):=\big\{c:[0,\infty)\to M \text{ piece-wise smooth curve}\big\}/\sim,$$
  where $\sim$ is equivalence of curves by reparametrization. We will refer to a class of curves $[c]\in C(M)$ simply by $c$, thinking of it as an unparametrized curve. For $c\in C(M)$, we define the \emph{Busemann function}
  $$b_c^g:M\to\mathbb R\cup\{\infty\}:
   x\mapsto\lim_{s\to\infty}
    \left(\int_0^s\|\dot c(t)\|_gdt-d_g\big(x,c(s)\big)\right),$$
 where $d_g$ is the distance with respect to $g$. $b_c$ is well defined, since it is clearly invariant under reparametrization of $c$.
 \end{defn}
 
 We reproduce some facts from \cite{flores2013gromov}:
 \begin{prop}[Proposition 4.15 in \cite{flores2013gromov}]\label{propBusemannBasics}\quad
 	\begin{itemize}
 		\item If $b_c^g(x)=\infty$ for some $x\in M$, then $b_c^g\equiv \infty$.
 		\item If $c\in C(M)$ has finite length, then there exists some $\bar x$ in the Cauchy completion $M_C^g$ of $M$ such that \begin{equation}\label{eqnBusemannCauchy}
 		 b_c^g(x)=\length_g(c)-d_g(x,\bar x).
 		\end{equation} 
 		Conversely, for every $\bar x\in M_C^g$, there exists a $c\in C(M)$ with finite length satisfying (\ref{eqnBusemannCauchy}).
 	\end{itemize}
 \end{prop}
 
 \begin{defn}
 	Let $B(M)^g:=\{b_c\}_{c\in C(M)}\backslash\{\infty\}$ be the set of \emph{finite Busemann functions}. The \emph{Busemann completion} of $(M,g)$ is given by $M_B^g:=B(M)^g/\mathbb R$, where $\mathbb R$ acts by the addition of constant functions.
 \end{defn}
 Note that the second part of proposition \ref{propBusemannBasics}, we can identify the Cauchy completion $M_C^g$ as a subset of the Busemann completion $M_B^g$ via the map $\bar x\mapsto[-d_g(\bar x,\cdot)]$.

 We define an equivalence relation on $C(M)$ by
 $$c_1 \sim_g c_2 \iff \exists r\in\mathbb R:b_{c_2}^g=b_{c_1}^g+r.$$
 The map $[c]\mapsto[b_c]$ gives us an isomorphism 
 $$C(M)/{\sim_g}\,\cong M_B^g\cup\{[c\in C(M):b_c^g\equiv \infty]\}.$$
 This allows us to interpret $M_B^g$ as a set of equivalence classes of curves. This will help us understand how $M_B$ depends on the metric since the curves are independent of the metric; only the equivalence relation changes.
 
 Let $g_1$ and $g_2$ be two Riemannian metrics on $M$. The rest of the appendix will be dedicated to proving the following proposition:
 
  \begin{prop}\label{propCauchyComparisons}
  	If $M_C^{g_2}$ is compact, then every point $\bar x\in M_C^{g_1}$ can be represented by a point $\bar x'\in M_C^{g_2}$, in the sense that there is a curve $c\in C(M)$ such that $[c]_{g_1}=\bar x$ and $[c]_{g_2}=\bar x'$.
  \end{prop}
  
  The following lemma will do most of the work for us:
   
  \begin{lemma}\label{lemLimGeodesic}
  	Let $(x_n)_{n=1}^\infty$ be a sequence in $M$ which $g_1$-converges to some $\bar x\in M_C^{g_1}$. Then there exists a $g_2$-geodesic curve $c$ such that $\bar x$ lies in the $g_1$-closure of the image of $c$. 
  \end{lemma}
  \begin{proof}
  	Let $p\in M$. For each $n\in\mathbb N$, let $c_n\in C(M)$ be a $g_2$-geodesic ray starting from $x_0$ and passing through $x_n$. Each $c_n$ is generated by the exponential mapping from a direction $u_n\in\mathbb S(T_pM)$ in the unit sphere bundle at $p$. Since $\mathbb S(T_pM)$ is compact, there exists a subsequence $(u_{n_k})_{k=1}^\infty$ of $(u_n)_{n=1}^\infty$ converging to some $u_\infty$. Let $c$ be the geodesic generated by $u_\infty$.
  	
  	Let $\epsilon>0$. Since $\lim_{k\to\infty}u_{n_k}=u_\infty$, there exists a $k_0\in\mathbb N$ such that $u_{n_k}\in B_\epsilon(u_\infty)$ for all $k>k_0$, where $B_\epsilon(u_\infty)$ is the $\epsilon$-ball around $u_\infty$ in $\mathbb S(T_pM)$. Consider the set
  	$$S_\epsilon :=\{\exp_p^{g_2}(tu):t\in [0,\infty),u\in B_\epsilon(N_\infty),tu\in\text{domain}(\exp_p^{g_2})\}\subseteq M.$$
  	For each $k>k_0$, $$S_\epsilon\supset \text{Image}(c_{n_k})\owns x_{n_k}.$$
  	Thus $$\overline{S_\epsilon}\supset\overline{\{x_{n_k}\}_{k>k_0}}\owns\bar x,$$
  	where $\overline{\cdot}$ indicates the closure in $M_C^{g_1}$. Since $\epsilon>0$ was arbitrarily chosen, we find that 
  	$$\bar x\in\bigcap_{\epsilon>0}\bar S_{\epsilon}=\overline{\text{Image}(c)}.$$ 
  \end{proof} 
 
 \begin{lemma}\label{lemDoubleLimGeodesic}
 	Let $g_3$ be a third metric on $M$. 
 	Let $(x_n)_{n=1}^\infty$ be a sequence in $M$ which $g_1$-converges to some $\bar x_1\in M_C^{g_1}$ and $g_3$-converges to some $\bar x_3\in M_C^{g_3}$. Then there exists a $g_2$-geodesic curve $c$ such that $\bar x_j$ lies in the $g_j$-closure of the image of $c$ for each $j\in\{1,3\}$.
 \end{lemma}
 \begin{proof}
 	This follows from applying the previous lemma twice (once using $g_3$ in place of $g_1$), and noting from the proof that resulting curve would be the same, since it doesn't depend on $g_1$.
 \end{proof}

  We would like to build an equivalence relation which includes the information of both $\sim_{g_1}$ and $\sim_{g_2}$. We define the relation $\sim_{g_2}^{g_1}$ by 
 $$c_1\sim_{g_2}^{g_1} c_2 \iff c_1\sim_{g_1} c_2 \text{ or }c_1\sim_{g_2} c_2.$$
 Note that  $\sim_{g_2}^{g_1}$ is not an equivalence relation in general, since it may not be transitive. To see this, consider the following example. Let $M=\mathbb S^1\times\mathbb S^1\times\mathbb R.$ Let $g_1$ (respectively $g_2$) be a metric which contracts the first (respectively second) $\mathbb S^1$ factor to a point at the origin of $\mathbb R$. For example, using coordinates $(\theta_1,\theta_2,r)\in\mathbb S^1\times\mathbb S^1\times\mathbb R$,
 \begin{align*}
  g_1&=r^2 d\theta_1^2 + d\theta_2^2+dr^2, \\
  g_2&=d\theta_1^2 + r^2d\theta_2^2+dr^2.
 \end{align*}
 Let $\{t_1,t_2,t_3\}$ be different points in $\mathbb S^1$. Consider the constant curves on $M:$
 \begin{align*}
  c_1 = (t_1,t_2,0), \\ c_2 = (t_3,t_2,0), \\ c_3 = (t_3,t_1,0).
 \end{align*}
 We see that $c_1\sim_{g_1} c_2\sim_{g_2} c_3$, but $c_1\not\sim^{g_1}_{g_2} c_3$.

 Consider the equivalence relation $\approx$ on $C(M)$ defined by $c_1\approx c_2$ if and only if there exists a finite set of curves $\{C_i\}_{i=1}^m\subset C(M)$ such that 
 $$c_1\sim_{g_2}^{g_1} C_1 \sim_{g_2}^{g_1} C_2 \sim_{g_2}^{g_1}\cdots\sim_{g_2}^{g_1} C_m\sim_{g_2}^{g_1} c_2.$$
 Define $M_B^\approx:= (C(M)/\approx)\backslash\{[c\in C(M):b_c^{g_1}=g_c^{g_2}\equiv \infty]\}.$ Let $\pi_i:M_B^{g_i}\to M_B^\approx$ be the natural quotient map for each $i\in\{1,2\}$.

 \begin{lemma}\label{lemCauchyInBusemann}
 	$\pi_1(M_C^{g_1})\subseteq\pi_2(M_B^{g_2}).$
 \end{lemma}
 \begin{proof}
  Assume that there exists $\bar x\in M_C^{g_1}$ such that $\pi_1(\bar x)\notin\pi_2(M_B^{g_2})$. Unravelling the definitions, this means that for each curve $c\in C(M)$ $g_1$-converging to $\bar x$, $b_c^{g_2}\equiv\infty$. Fix $c(t)$ a parametrized curve $g_1$-converging to $\bar x$. Since $b_c^{g_2}\equiv\infty$, $c$ has infinite $g_2$-length by proposition \ref{propBusemannBasics}.
  
  \textbf{Claim:} The image of $c$ is unbounded with respect to $g_2$.
  \begin{proof}[Proof of claim]
  	Assume that the claim is false. Then the image of $c$ lies in some closed set $B\subset M_C^{g_2}$ which is bounded with respect to $g_2$. Thus there exists a sequence $(t_n)_{n=1}^\infty$ in $\mathbb R$ such that $\big(c(t_n)\big)_{n=1}^\infty$ converges in $g_2$ to some $\bar x'\in B$. Applying lemma \ref{lemDoubleLimGeodesic} we find a $g_2$ geodesic curve $c'$ which $g_1$-converges to $\bar x$ and $g_2$ converges to $\bar x'$. Since $c'$ $g_2$-converges to $\bar x'$ it has finite $g_2$-length. Thus $b_{c'}^{g_2}\not\equiv\infty$, contradicting the assumption on $\bar x$ since $c'$ $g_1$-converges to $\bar x$.
  \end{proof}
  The claim allows us to find a sequence $(t_n)_{n=1}^\infty$ in $\mathbb R$ such that for any $x_0\in M$, $\lim_{n\to\infty}d_{g_2}\big(x_0,c(t_n)\big)=\infty$. Applying lemma \ref{lemLimGeodesic} to this sequence gives a $g_2$-geodesic curve $c'$ which $g_1$-converges to $\bar x$. From \cite{sakai1996riemannian} the function
  $$t:\mathbb S(T_{x_0}M)\to(0,\infty]:u\mapsto\sup\left\{t>0:d_{g_2}(x_0,\exp_{x_0}tu)=t\right\}$$
  is continuous. Thus $$t(u_\infty)=\lim_{k\to\infty}t(u_{n_k})\geq\lim_{k\to\infty}d_{g_2}\big(x_0,c(t_{n_k})\big)=\infty,$$
  so that $t(u_\infty)=\infty,$ where $(u_{n_k})_{k=1}^\infty$ is the sequence constructed in lemma \ref{lemLimGeodesic}. In other words, the geodesic $c'$ generated by $u_\infty$ is $g_2$-distance minimizing. Thus
  $$b_{c'}^{g_2}(x_0) = \lim_{s\to\infty}\left(
   \int_0^s\|\dot{c}_\infty(t)\|_{g_2}dt - d_{g_2}(x_0,c_\infty(s))
   \right) =\lim_{s\to\infty}0=0,$$
  so that $b_{c'}^{g_2}\not\equiv\infty.$ This contradicts the assumption on $\bar x$, since $c'$ $g_1$-converges to $\bar x$.
 \end{proof}
 
 \begin{lemma}
 	If $M_C^{g}$ is compact, then $M_C^{g}=M_B^{g}$.
\end{lemma}
 \begin{proof}
 	Let $x_0\in M$. Since $M_C^g$ is compact, the diameter of $M$ is bounded, so that the function $d_g(x_0,\cdot)$ is bounded. This implies that $b_c^g(x_0)=\infty$ for all infinite length curves $c$. The result then follows from proposition \ref{propBusemannBasics}.
 \end{proof}
 
 Combining the previous two lemmas produces a proof of proposition \ref{propCauchyComparisons}.

\newcommand{\etalchar}[1]{$^{#1}$}


\end{document}